\documentclass[10pt,a4paper,reqno]{amsart}
\usepackage[latin9]{inputenc}
\usepackage[active]{srcltx}
\usepackage{amstext}
\usepackage{amsthm}
\usepackage{amssymb}

\makeatletter

\usepackage{amsthm}
\usepackage{graphicx}
\usepackage{color}

\usepackage[font=small]{caption}

\usepackage{multirow}
\usepackage[shortlabels]{enumitem}

\theoremstyle{plain}
\newtheorem{thm}{Theorem}\theoremstyle{definition}
\theoremstyle{remark}
\newtheorem{rem}[thm]{Remark}\theoremstyle{plain}
\newtheorem{prop}[thm]{Proposition}\theoremstyle{plain}
\newtheorem{lem}[thm]{Lemma}\theoremstyle{plain}
\theoremstyle{definition}
\theoremstyle{remark}
\newtheorem*{rem*}{Remark}\theoremstyle{definition}

\usepackage{amsfonts}
\usepackage{mathrsfs}

\theoremstyle{plain}

\newcommand{\N}{\mathbb{N}}
\newcommand{\R}{{\mathbb{R}}}
\newcommand{\C}{{\mathbb{C}}}

\newcommand{\ii}{{\rm i}}

\newcommand{\diag}{\mathop\mathrm{diag}\nolimits}
\newcommand{\spn}{\mathop\mathrm{span}\nolimits}

\renewcommand{\Re}{\mathop\mathrm{Re}\nolimits}
\renewcommand{\Im}{\mathop\mathrm{Im}\nolimits}

\newcommand{\rank}{\mathop\mathrm{rank}\nolimits}

\newcommand{\Res}{\mathop\mathrm{Res}\nolimits}

\newcommand{\Ch}{\mathop\mathrm{Comm}_{H}\nolimits}

\makeatother

\begin{document}
\title{On Hankel matrices commuting with Jacobi matrices from the Askey scheme}
\author{František Štampach}
\address{{[}František Štampach{]} Department of Applied Mathematics, Faculty
of Information Technology, Czech Technical University in~Prague,
Thákurova~9, 160~00 Praha, Czech Republic}
\email{stampfra@fit.cvut.cz}
\author{Pavel Š\v{t}oví\v{c}ek}
\address{{[}Pavel Š\v{t}oví\v{c}ek{]} Department of Mathematics, Faculty of
Nuclear Sciences and Physical Engineering, Czech Technical University
in Prague, Trojanova~13, 12000 Praha, Czech Republic}
\email{stovicek@fjfi.cvut.cz}
\begin{abstract}
A complete characterization is provided of Hankel matrices commuting
with Jacobi matrices which correspond to hypergeometric orthogonal
polynomials from the Askey scheme. It follows, as the main result
of the paper, that the generalized Hilbert matrix is the only prominent
infinite-rank Hankel matrix which, if regarded as an operator on $\ell^{2}(\mathbb{N}_{0})$,
is diagonalizable by application of the commutator method with Jacobi
matrices from the mentioned families.
\end{abstract}

\date{\today}
\maketitle

\section{Introduction and motivation}

Operators acting on~$\ell^{2}(\N_{0})$ determined by a Hankel matrix
\begin{equation}
\mathcal{H}=\begin{pmatrix}h_{0} & h_{1} & h_{2} & h_{3} & \dots\\
h_{1} & h_{2} & h_{3} & h_{4} & \dots\\
h_{2} & h_{3} & h_{4} & h_{5} & \dots\\
h_{3} & h_{4} & h_{5} & h_{6} & \dots\\
\vdots & \vdots & \vdots & \vdots & \ddots
\end{pmatrix}\!,\label{eq:def_hankel}
\end{equation}
with $h_{k}\in\R$, belong to one of the basic classes of linear operators
whose general theory has been systematically developed during the
last century~\cite{peller_03,power_82}. In contrast to other well
known classes such as Jacobi, Schr\"odinger, Toeplitz operators,
etc., there are only very few known concrete Hankel matrices that
admit an explicit diagonalization, i.e, whose spectral problem is
explicitly solvable. In fact, concerning Hankel matrices of infinite
rank, the authors were aware of only one such example - the famous
(generalized) Hilbert matrix - until recently.

The (generalized) Hilbert matrix is the Hankel matrix~\eqref{eq:def_hankel}
with
\[
h_{k}=\frac{1}{k+t},\quad k\in\N_{0},
\]
where $t\notin-\N_{0}$ is a parameter. The Hilbert matrix determines
a bounded operator $H_{t}$ on~$\ell^{2}(\N_{0})$ which can be,
quite surprisingly, diagonalized fully explicitly. In full generality,
the diagonalization of $H_{t}$ was obtained by Rosenblum in \cite{rosenblum_pams58a,rosenblum_pams58b}.
Preliminary results appeared even earlier \cite{magnus_ajm50,shanker_pcps49},
however. Rosenblum's original approach relies on an idea of the so-called
commutator method that was successfully applied to other operators
later on, see~\cite{kalvoda-stovicek_lma16,stampach-stovicek_jmaa19,yafaev_fap10}.

The commutator method relies on two main steps. First, one has to
find a commuting operator with simple spectrum and explicitly solvable
spectral problem and, second, determine a spectral mapping. In the
case of the Hilbert matrix, one may prefer to consider the transformed
Hankel matrix with the entries
\begin{equation}
\tilde{h}_{k}=\frac{(-1)^{k}}{k+t}\,,\,\ k\in\N_{0}.\label{eq:def_tilde_h}
\end{equation}
Obviously, the corresponding operator~$\tilde{H}_{t}$ is unitarily
equivalent to~$H_{t}$ via unitary transform $U=\diag(1,-1,1,-1,\dots)$.
The operator $\tilde{H}_{t}$ can be readily shown to commute with
a self-adjoint Jacobi operator $J_{t}$ given by a tridiagonal matrix
$\mathcal{J}_{t}$ with the entries
\begin{equation}
\left(\mathcal{J}_{t}\right)_{n,n}=2n(n+t),\quad\left(\mathcal{J}_{t}\right)_{n,n+1}=\left(\mathcal{J}_{t}\right)_{n+1,n}=(n+1)(n+t),\quad n\in\N_{0},\label{eq:def_J_t}
\end{equation}
which is the first ingredient for the commutator method to be applied.
Indeed, the spectrum of a self-adjoint Jacobi operator is always simple.
Moreover, up to an inessential addition of a multiple of the identity
operator, the Jacobi operator~$J_{t}$ corresponds to a particular
subfamily of orthogonal polynomials known as the continuous dual Hahn
polynomials listed in the Askey hypergeometric scheme~\cite{koekoek-etal_10}.
As a result, the spectral problem of~$J_{t}$ is explicitly solvable.
Next, by a general fact, $\tilde{H}_{t}=h(J_{t})$ for a Borel function
$h$. The second step of the commutator method is to determine $h$
which can be done with the aid of a known generating function formula
for the continuous dual Hahn polynomials, see~\cite{kalvoda-stovicek_lma16}
for details. The desired diagonalization of~$\tilde{H}_{t}$, and
hence also of $H_{t}$, then follows.

Of course, relying on the commutator method may seem rather restrictive.
On the other hand, whenever the method turned out to be applicable
it proved itself to be very powerful. In particular, as far as the
Hilbert matrix is concerned, the authors are not aware of an alternative
way of its diagonalization, other than a variant of the commutator
method.

The Askey scheme of hypergeometric orthogonal polynomials~\cite{koekoek-etal_10}
can be viewed as a rich source of Jacobi operators with an explicitly
solvable spectral problem. For instance, the Wilson polynomials standing
at the top of the Askey scheme depend on four additional parameters
and hence the same is true for the corresponding Jacobi matrix. Our
initial goal was to explore thoroughly this reservoir while trying
to find out whether the (generalized) Hilbert matrix is the only Hankel
matrix which can be diagonalized with the aid of the commutator method
with Jacobi matrices from the Askey scheme or whether there are other
infinite-rank Hankel matrices having such a property. And this question
was the main motivation for the current paper.

In this study, we distinguish nine families of Hermitian non-decomposable
Jacobi matrices
\begin{equation}
\mathcal{J}=\begin{pmatrix}\beta_{0} & \alpha_{0}\\
\alpha_{0} & \beta_{1} & \alpha_{1}\\
 & \alpha_{1} & \beta_{2} & \alpha_{2}\\
 &  & \alpha_{2} & \beta_{3} & \alpha_{3}\\
 &  &  & \ddots & \ddots & \ddots
\end{pmatrix}\!,\label{eq:def_jacobi}
\end{equation}
with $\alpha_{n}>0$ and $\beta_{n}\in\R$, corresponding to hypergeometric
orthogonal polynomials from the Askey scheme whose measure of orthogonality
is positive and supported on an infinite subset of~$\R$. These comprise
Wilson, Continuous dual Hahn, Continues Hahn, Jacobi, Meixner--Pollaczek,
Meixner, Laguerre, Charlier, and Hermite polynomials. On the other
hand, Racah, Hahn, dual Hahn, Krawtchouk, Pseudo-Jacobi, and Bessel
polynomials are excluded for the associated Jacobi matrices are finite
or non-Hermitian. For each family, our main theorem determines the
set of parameters for which there exists a nonzero Hankel matrix commuting
with the Jacobi matrix in question. In addition, in these quite rare
situations, the space of commuting Hankel matrices is always of dimension~$2$ with an explicitly described basis; see Theorem~\ref{thm:main}
below.

It may also seem that the set of Jacobi matrices, which we restrict
our analysis to, is too special. However, from the view point of the
commutator method, whose principle ingredient is always a commuting
operator with an explicitly solvable spectral problem, this restriction
is reasonable. Naturally, it would be really interesting to have a
general characterization of Jacobi matrices commuting with a nontrivial
Hankel matrix. Such a problem was solved by Gr\"unbaum in \cite{grunbaum_laa81}
for finite Jacobi matrices commuting with Toeplitz matrices. However,
as far as semi-infinite Jacobi and Hankel matrices are concerned,
a solution of this problem seems to be out of reach at the moment.

This paper deals with the commutation equation with semi-infinite
matrices on the algebraic level only. Passing to operators acting
on~$\ell^{2}(\N_{0})$, one would have to additionally require the
columns (and hence the rows) of a Hankel matrix~$\mathcal{H}$ to
be square summable. In that case~$\mathcal{H}$ defines a densely
defined operator on~$\ell^{2}(\N_{0})$ with the canonical basis
in its domain. A consequence of our main theorem is that the (generalized)
Hilbert matrix is the only infinite-rank Hankel matrix which can be
diagonalized by application of the commutator method using Hermitian
non-decomposable Jacobi matrices from the Askey scheme. More precisely,
an immediate consequence of Theorem~1 is as follows: \vskip3pt \emph{Let
$\mathcal{H}$ be a Hankel matrix with $\ell^{2}$-columns and $\rank\mathcal{H}>1$
and $\mathcal{J}$ any Hermitian non-decomposable Jacobi matrix from
the Askey scheme (concretely specified below). If \linebreak $\mathcal{H}\mathcal{J}=\mathcal{J}\mathcal{H}$
then $\mathcal{H}$ is a scalar multiple of the Hankel matrix~$\tilde{\mathcal{H}}_{t}$
with entries determined in \eqref{eq:def_tilde_h} for some $t$.}
\vskip3pt This statement is an answer to our original question and
emphasizes even more the prominent role played by the Hilbert matrix.

In a bit more detail, our analysis reveals, too, that in all cases
when the prominent Hankel matrix~$\tilde{\mathcal{H}}_{t}$ commutes
with a Jacobi matrix $\mathcal{J}$ corresponding to orthogonal polynomials
from the Askey scheme the diagonal and the off-diagonal sequence of
$\mathcal{J}$ are polynomial functions of index of order 2 and 4,
respectively. There is another family of orthogonal polynomials known
as the Stieltjes--Carlitz polynomials with the corresponding Jacobi
matrix of the same type but not included in the hypergeometric Askey
scheme. Rather than to hypergeometric series these polynomials are
intimately related to the Jacobian elliptic functions. This observation
led the authors to apply the commutator method to Jacobi operators
associated with the Stieltjes--Carlitz polynomials. This study resulted
in a discovery of four new explicitly diagonalizable Hankel matrices.
Their diagonalization is treated in a separate paper~\cite{stampach-stovicek_inprep}.

The article is organized as follows. In Section~2, the basic definitions
of the considered Jacobi matrices from the Askey scheme as well as
restrictions on the involved parameters are recalled. Then the main
result is formulated in Theorem~\ref{thm:main}. Section~3 is devoted
to an auxiliary result used in several particular proofs, especially
in those corresponding to the most complicated cases. In Section~\ref{sec:proofs},
a proof of the main result is given. The mentioned families of orthogonal
polynomials are treated case by case starting from the most complicated
one (Wilson) and proceeding to simpler cases following the hierarchy
of the Askey scheme. A more sophisticated approach is used for the
first three families depending on 3 or 4 additional parameters while
the remaining cases are treated using a rather elementary computational
approach.

\section{Notation and statement of the main result}

\label{sec:main} In the notation, the Jacobi matrix associated with
a particular family of orthogonal polynomials is distinguished by
a corresponding superscript, i.e., we consider nine families of Jacobi
matrices
\[
\mathcal{J}=\mathcal{J}^{j},\quad\mbox{ for }j\in\{\text{W},\text{CdH},\text{CH},\text{J},\text{MP},\text{M},\text{L},\text{C},\text{H}\}.
\]
The entries of these matrices can depend on up to four parameters.
For the chosen parametrization, we strictly follow~\cite{koekoek-etal_10}.
The concrete definitions of the entries are listed below.

\vskip8pt I) The Jacobi matrix~$\mathcal{J}^{\text{W}}$ associated
with the \textbf{Wilson} polynomials:
\begin{equation}
\alpha_{n}^{\text{W}}=\alpha_{n}^{\text{W}}(a,b,c,d):=\sqrt{A_{n}^{\text{W}}C_{n+1}^{\text{W}}}\quad\mbox{ and }\quad\beta_{n}^{\text{W}}=\beta_{n}^{\text{W}}(a,b,c,d):=A_{n}^{\text{W}}+C_{n}^{\text{W}}-a^{2},\label{eq:def_alpha_beta_wilson}
\end{equation}
where
\begin{align}
A_{n}^{\text{W}}=A_{n}^{\text{W}}(a,b,c,d) & :=\frac{(n+a+b)(n+a+c)(n+a+d)(n+a+b+c+d-1)}{(2n+a+b+c+d-1)(2n+a+b+c+d)}\,,\label{eq:def_A_wilson}\\
C_{n}^{\text{W}}=C_{n}^{\text{W}}(a,b,c,d) & :=\frac{n(n+b+c-1)(n+b+d-1)(n+c+d-1)}{(2n+a+b+c+d-2)(2n+a+b+c+d-1)}\,.\label{eq:def_C_wilson}
\end{align}
Though not obvious at first glance in the case of $\beta_{n}^{\text{W}}$,
the coefficients $\alpha_{n}^{\text{W}}$ and $\beta_{n}^{\text{W}}$
are both symmetric functions of the parameters $a,b,c,d$. The parameters
$a,b,c,d\in\C$ are restricted so that one of the following holds: 
\begin{enumerate}
\item[1.]  $\Re a,\Re b,\Re c,\Re d>0$ and the non-real parameters occur in
complex conjugate pairs. 
\item[2.]  Up to a permutation of parameters, $a<0$ and $a+b$, $a+c$, $a+d$
positive or a pair of complex conjugates occur with positive real
parts. 
\end{enumerate}
\noindent \vskip8pt II) The Jacobi matrix~$\mathcal{J}^{\text{CdH}}$
associated with the \textbf{Continuous dual Hahn} polynomials:
\[
\alpha_{n}^{\text{CdH}}=\alpha_{n}^{\text{CdH}}(a,b,c):=\sqrt{A_{n}^{\text{CdH}}C_{n+1}^{\text{CdH}}}\quad\mbox{ and }\quad\beta_{n}^{\text{CdH}}=\beta_{n}^{\text{CdH}}(a,b,c):=A_{n}^{\text{CdH}}+C_{n}^{\text{CdH}}-a^{2},
\]
where
\begin{align*}
A_{n}^{\text{CdH}}=A_{n}^{\text{CdH}}(a,b,c) & :=(n+a+b)(n+a+c)\,,\\
C_{n}^{\text{CdH}}=C_{n}^{\text{CdH}}(a,b,c) & :=n(n+b+c-1)\,.
\end{align*}
Hence
\[
\alpha_{n}^{\text{CdH}}=\sqrt{(n+1)(n+a+b)(n+a+c)(n+b+c)}
\]
and
\[
\beta_{n}^{\text{CdH}}=2n^{2}+(2a+2b+2c-1)n+ab+ac+bc.
\]
Clearly, both $\alpha_{n}^{\text{CdH}}$ and $\beta_{n}^{\text{CdH}}$
are symmetric functions of~$a,b,c$. The parameters $a,b,c\in\C$
are restricted such that one of the following holds:
\begin{enumerate}
\item[1.]  $a,b,c$ are positive except possibly a pair of complex conjugates
with positive real parts. 
\item[2.]  Up to a permutation of parameters, $a<0$ and $a+b$, $a+c$, are
positive or a pair of complex conjugates with positive real parts.
\end{enumerate}
\noindent \vskip8pt III) The Jacobi matrix~$\mathcal{J}^{\text{CH}}$
associated with the \textbf{Continuous Hahn} polynomials:
\[
\alpha_{n}^{\text{CH}}=\alpha_{n}^{\text{CH}}(a,b,c,d):=\sqrt{-A_{n}^{\text{CH}}C_{n+1}^{\text{CH}}}\quad\mbox{ and }\quad\beta_{n}^{\text{CH}}=\beta_{n}^{\text{CH}}(a,b,c,d):=\mathrm{i}(A_{n}^{\text{CH}}+C_{n}^{\text{CH}}+a),
\]
where
\begin{align*}
A_{n}^{\text{CH}}=A_{n}^{\text{CH}}(a,b,c,d) & :=-\frac{(n+a+b+c+d-1)(n+a+c)(n+a+d)}{(2n+a+b+c+d-1)(2n+a+b+c+d)}\,,\\
C_{n}^{\text{CH}}=C_{n}^{\text{CH}}(a,b,c,d) & :=\frac{n(n+b+c-1)(n+b+d-1)}{(2n+a+b+c+d-2)(2n+a+b+c+d-1)}\,.
\end{align*}
In this case, $\alpha_{n}^{\text{CH}}$ and $\beta_{n}^{\text{CH}}$
are not symmetric functions of $a,b,c,d$ but remain invariant under
the permutations of $(a,b,c,d)$ which are composed of two $2$-cycles
($(b,a,d,c)$, $(c,d,a,b)$, $(d,c,b,a)$). The parameters are supposed
to be such that
\begin{equation}
c=\bar{a},\;d=\bar{b}\quad\mbox{ and }\quad\Re a>0,\;\Re b>0.\label{eq:param_assum_chahn}
\end{equation}
Then $-A_{n}^{\text{CH}}C_{n+1}^{\text{CH}}>0$ and $\beta_{n}\in\mathbb{R}$
for all $n\in\N_{0}$, indeed.

\vskip8pt IV) The Jacobi matrix $\mathcal{J}^{\text{J}}$ associated
with the \textbf{Jacobi} polynomials:
\begin{equation}
\alpha_{n}^{\text{J}}=\alpha_{n}^{\text{J}}(\alpha,\beta):=\sqrt{\frac{4(n+1)(n+\alpha+1)(n+\beta+1)(n+\alpha+\beta+1)}{(2n+\alpha+\beta+1)(2n+\alpha+\beta+2)^{2}(2n+\alpha+\beta+3)}}\label{eq:def_alpha_jacobi}
\end{equation}
and
\begin{equation}
\beta_{n}^{\text{J}}=\beta_{n}^{\text{J}}(\alpha,\beta):=\frac{\beta^{2}-\alpha^{2}}{(2n+\alpha+\beta)(2n+\alpha+\beta+2)},\label{eq:def_beta_jacobi}
\end{equation}
where $\alpha,\beta>-1$.

\noindent \vskip8pt V) The Jacobi matrix $\mathcal{J}^{\text{MP}}$
associated with the \textbf{Meixner--Polaczek} polynomials:
\begin{equation}
\alpha_{n}^{\text{MP}}=\alpha_{n}^{\text{MP}}(\lambda,\phi):=\frac{\sqrt{(n+1)(n+2\lambda)}}{2\sin\phi}\quad\mbox{ and }\quad\beta_{n}^{\text{MP}}=\beta_{n}^{\text{MP}}(\lambda,\phi):=-\frac{n+\lambda}{\tan\phi},\label{eq:def_alpha_beta_meixner-pollaczek}
\end{equation}
where $\lambda>0$ and $\phi\in(0,\pi)$. For $\phi=\pi/2$, $\beta_{n}^{\text{MP}}$
is to be understood as zero for all $n\in\N_{0}$.

\noindent \vskip8pt VI) The Jacobi matrix $\mathcal{J}^{\text{M}}$
associated with the \textbf{Meixner} polynomials:
\begin{equation}
\alpha_{n}^{\text{M}}=\alpha_{n}^{\text{M}}(c,\beta):=\frac{\sqrt{c(n+1)(n+\beta)}}{1-c}\quad\mbox{ and }\quad\beta_{n}^{\text{M}}=\beta_{n}^{\text{M}}(c,\beta):=\frac{n+(n+\beta)c}{1-c},\label{eq:def_alpha_beta_meixner}
\end{equation}
where $\beta>0$ and $c\in(0,1)$.

\noindent \vskip8pt VII) The Jacobi matrix $\mathcal{J}^{\text{L}}$
associated with the \textbf{Laguerre} polynomials:
\begin{equation}
\alpha_{n}^{\text{L}}=\alpha_{n}^{\text{L}}(\alpha):=\sqrt{(n+1)(n+\alpha+1)}\quad\mbox{ and }\quad\beta_{n}^{\text{L}}=\beta_{n}^{\text{L}}(\alpha):=2n+\alpha+1,\label{eq:def_alpha_beta_laguerre}
\end{equation}
where $\alpha>-1$.

\noindent \vskip8pt VIII) The Jacobi matrix $\mathcal{J}^{\text{C}}$
associated with the \textbf{Charlier} polynomials:
\begin{equation}
\alpha_{n}^{\text{C}}=\alpha_{n}^{\text{C}}(a):=\sqrt{a(n+1)}\quad\mbox{ and }\quad\beta_{n}^{\text{C}}=\beta_{n}^{\text{C}}(a):=n+a,\label{eq:def_alpha_beta_charlier}
\end{equation}
where $a>0$.

\noindent \vskip8pt IX) The Jacobi matrix $\mathcal{J}^{\text{H}}$
associated with the \textbf{Hermite} polynomials:
\begin{equation}
\alpha_{n}^{\text{H}}=\sqrt{\frac{n+1}{2}}\quad\mbox{ and }\quad\beta_{n}^{\text{H}}=0.\label{eq:def_alpha_beta_hermite}
\end{equation}

\noindent Note that for any $j\in\{\text{W},\text{CdH},\text{CH},\text{J},\text{MP},\text{M},\text{L},\text{C},\text{H}\}$,
$\alpha_{n}^{j}$ and $\beta_{n}^{j}$ may be both regarded as analytic
functions in~$n$ on a neighborhood of infinity. Hence, whenever
convenient, $\alpha_{n}^{j}$ and $\beta_{n}^{j}$ are extended to
non-integer values of~$n$.

Finally, by a \emph{commutant} of~$\mathcal{J}$, we mean the space
of semi-infinite matrices $\mathcal{A}\in\C^{\infty,\infty}$ commuting
with~$\mathcal{J}$, i.e, $\mathcal{AJ}=\mathcal{JA}$ entrywise.
Since $\mathcal{J}$ is banded the matrix products $\mathcal{AJ}$
and $\mathcal{JA}$ are well defined for any $\mathcal{A}\in\C^{\infty,\infty}$.
The special attention is paid to Hankel matrices commuting with a
given Jacobi matrix~$\mathcal{J}$, therefore we define the \emph{Hankel
commutant} of~$\mathcal{J}$ by
\[
\Ch(\mathcal{J}):=\{\mathcal{H}\in\C^{\infty,\infty}\mid\mathcal{JH}=\mathcal{HJ}\mbox{ and }\mathcal{H}\mbox{ is a Hankel matrix}\}.
\]
Clearly, $\Ch(\mathcal{J})$ is a linear subspace of~$\C^{\infty,\infty}$.

Our main result characterizes $\Ch(\mathcal{J}^{j})$ for $j\in\{\text{W},\text{CdH},\text{CH},\text{J},\text{MP},\text{M},\text{L},\text{C},\text{H}\}$.
Only in very particular cases, $\Ch(\mathcal{J}^{j})$ is nontrivial,
i.e., there exists a nonzero Hankel matrix commuting with a Jacobi
matrix from the Askey scheme.

\renewcommand{\theenumi}{\roman{enumi}}%

\begin{thm}\label{thm:main} For any $j\in\{\text{W},\text{CdH},\text{CH},\text{J},\text{MP},\text{M},\text{L},\text{C},\text{H}\}$,
the Hankel commutant\linebreak $\Ch(\mathcal{J}^{j})$ is either
trivial or a two-dimensional space. Moreover, $\dim\Ch(\mathcal{J}^{j})=2$
if and only if $\alpha_{n}^{j}$ and $\beta_{n}^{j}$ depend polynomially
on~$n$ and $\alpha_{-1}^{j}=0$.

Below, we characterize all situations when $\Ch(\mathcal{J}^{j})$
is nontrivial by specifying parameters case by case. In each case,
with $\Ch(\mathcal{J}^{j})$ being nontrivial, we determine two sequences,
$h^{(1)}$ and $h^{(2)}$, which define two Hankel matrices (as in~\eqref{eq:def_hankel})
that form a basis of $\Ch(\mathcal{J}^{j})$.
\begin{enumerate}
\item $\dim\Ch(\mathcal{J}^{\text{W}})=2$ if and only if, up to a permutation,
the parameters $a,b,c,d$ fulfill
\[
a=\frac{3}{4},\quad b=\frac{t}{2}+\frac{1}{4},\quad c=\frac{t}{2}-\frac{1}{4},\quad d=\frac{1}{4},
\]
for $t>0$, and if that is the case, we have
\[
h_{k}^{(1)}=\frac{(-1)^{k}}{k+t}\quad\mbox{ and }\quad h_{k}^{(2)}=(-1)^{k},\quad k\in\N_{0}.
\]
\item $\dim\Ch(\mathcal{J}^{\text{CdH}})=2$ if and only if, up to a permutation,
the parameters $a,b,c$ fulfill
\[
a=b=\frac{1}{2},\quad c=t-\frac{1}{2},
\]
for $t>0$, and $h^{(1)}$ and $h^{(2)}$ are as in the case~(i).
\item $\dim\Ch(\mathcal{J}^{\text{CH}})=2$ if and only if the parameters
$a,b,c,d$ fulfill
\[
a=\bar{c}=\frac{1}{4}+\ii t,\quad b=\bar{d}=\frac{3}{4}+\ii t,\quad\mbox{ or }\quad a=\bar{c}=\frac{3}{4}+\ii t,\quad b=\bar{d}=\frac{1}{4}+\ii t,
\]
for $t\in\R$, and if that is the case, we have
\[
h_{k}^{(1)}=\sin\frac{k\pi}{2}\quad\mbox{ and }\quad h_{k}^{(2)}=\cos\frac{k\pi}{2},\quad k\in\N_{0}.
\]
\item $\Ch(\mathcal{J}^{\text{J}})$ is trivial for all $\alpha,\beta>-1$.
\item $\dim\Ch(\mathcal{J}^{\text{MP}})=2$ if and only if $\lambda=1/2$
and $\phi\in(0,\pi)$ arbitrary. If so we have 
\[
h_{k}^{(1)}=\sin(k\phi)\quad\mbox{ and }\quad h_{k}^{(2)}=\cos(k\phi),\quad k\in\N_{0}.
\]
\item $\dim\Ch(\mathcal{J}^{\text{M}})=2$ if and only if $\beta=1$ and
$c\in(0,1)$ arbitrary. If so we have 
\[
h_{k}^{(1)}=(-1)^{k}c^{k/2}\quad\mbox{ and }\quad h_{k}^{(2)}=(-1)^{k}c^{-k/2},\quad k\in\N_{0}.
\]
\item $\dim\Ch(\mathcal{J}^{\text{L}})=2$ if and only if $\alpha=0$ and
we have
\[
h_{k}^{(1)}=(-1)^{k}\quad\mbox{ and }\quad h_{k}^{(2)}=(-1)^{k}k,\quad k\in\N_{0}.
\]
\item $\Ch(\mathcal{J}^{\text{C}})$ is trivial for all $a>0$.
\item $\dim\Ch(\mathcal{J}^{\text{H}})=2$ and $h^{(1)}$ and $h^{(2)}$
are as in the case~(iii).
\end{enumerate}
\end{thm}

\renewcommand{\theenumi}{\arabic{enumi}}%

\begin{rem} Clearly, those restrictions of parameters which are dictated
by requiring the Jacobi matrices from the Askey scheme to be Hermitian
are not essential for every claim of Theorem~\ref{thm:main}. For
example, it follows from the claim~(ii) that the Jacobi matrix~$\mathcal{J}_{t}$
defined in~\eqref{eq:def_J_t} and the Hankel matrix~$\tilde{\mathcal{H}}_{t}$
whose entries are given by~\eqref{eq:def_tilde_h} commute for any
parameter $t>0$. Obviously this assertion can be extended to all
$t\in\C\setminus\left(-\N_{0}\right)$. \end{rem}

\section{An auxiliary result}

\label{sec:aux} The following lemma will be used repeatedly below.
Its statement can be of independent interest. It has been also applied
in~\cite{stampach-stovicek_inprep}.

\begin{lem}\label{lem:Mzw} Let $p$ and $q$ be complex functions
which are meromorphic in a neighborhood of $\infty$ and assume that
the order of the pole at $\infty$ equals $2$ for both of them. Further
let $\epsilon\in\mathbb{C}$, $\epsilon\neq0$, and put, for $z,w\in\mathbb{C}$
sufficiently large,
\[
M(z,w):=\left(\begin{array}{cc}
p(z+\epsilon)-p(w-\epsilon) & q(z+\epsilon)-q(w-\epsilon)\\
p(z-\epsilon)-p(w+\epsilon) & q(z-\epsilon)-q(w+\epsilon)
\end{array}\right)\!.
\]
Let us write the determinant of $M(z,w)$ in the form 
\[
\det M(z,w)=\left((z-w)^{2}-4\epsilon^{2}\right)\delta(z,w).
\]
If at least one of the functions $p(z)$ and $q(z)$ is not a polynomial
in $z$ of degree $2$ and the set of functions $\{1,p,q\}$ is linearly
independent, then one of the following two cases happens:

(i) for every $w\in\mathbb{C}$ sufficiently large there exists $\underset{z\to\infty}{\text{lim}}\delta(z,w)\in\mathbb{C}\setminus\{0\}$,

(ii) for every $w\in\mathbb{C}$ sufficiently large there exists $\underset{z\to\infty}{\text{lim}}z\delta(z,w)\in\mathbb{C}\setminus\{0\}$.\\
 Consequently, for every $w\in\mathbb{C}$ sufficiently large there
exists $R(w)>0$ such that for all $z\in\mathbb{C}$, $\left|z\right|>R(w)$,
the matrix $M(z,w)$ is regular. \end{lem}

\begin{proof} Without loss of generality we can assume that the leading
coefficient in the Laurent expansion about $\infty$ equals $1$ for
both $p$ and $q$. So we can write
\[
p(z)=z^{2}+\sum_{k=-1}^{\infty}p_{k}z^{-k},\ q(z)=z^{2}+\sum_{k=-1}^{\infty}q_{k}z^{-k}.
\]
Let
\[
P(x,y):=\sum_{n=1}^{\infty}p_{n}\sum_{j=0}^{n-1}x^{j}y^{n-j-1},\ Q(x,y):=\sum_{n=1}^{\infty}q_{n}\sum_{j=0}^{n-1}x^{j}y^{n-j-1},
\]
and
\[
P_{0}(y):=P(0,y)=\sum_{n=1}^{\infty}p_{n}y^{n-1},\ Q_{0}(y):=Q(0,y)=\sum_{n=1}^{\infty}q_{n}y^{n-1}.
\]
Note that
\begin{equation}
\frac{p(x)-p(y)}{x-y}=x+y+p_{-1}-\frac{1}{xy}\,P\!\left(\frac{1}{x},\frac{1}{y}\right)\!,\ \frac{q(x)-q(y)}{x-y}=x+y+q_{-1}-\frac{1}{xy}\,Q\!\left(\frac{1}{x},\frac{1}{y}\right)\!.\label{eq:p-diff-P}
\end{equation}

Clearly, $p$ and $q$ are polynomials of degree $2$ if and only
if $P_{0}(y)=Q_{0}(y)=0$. Furthermore, the functions $\{1,p,q\}$
are linearly dependent if and only if $p(z)-p_{0}=q(z)-q_{0}$ for
all sufficiently large $z$, and this happens if and only if $p_{-1}=q_{-1}$
and $P_{0}(y)=Q_{0}(y)$ (the latter equation can be equivalently
replaced by $P(x,y)=Q(x,y)$).

Thus we assume that either (1) $P_{0}(y)\neq Q_{0}(y)$ or (2) $P_{0}(y)=Q_{0}(y)\neq0$
(equivalently, $P(x,y)=Q(x,y)\neq0$) and $p_{-1}\neq q_{-1}$.

Put
\[
\tilde{\delta}(u,v):=\delta\!\left(\frac{1}{u},\frac{1}{v}\right)\!.
\]
We have to explore the properties of $\tilde{\delta}(u,v)$ for $u$
and $v$ sufficiently small. A straightforward computation based on
(\ref{eq:p-diff-P}) shows that
\begin{eqnarray}
 &  & \hskip-1.5em(1-\epsilon^{2}u^{2})(1-\epsilon^{2}v^{2})\tilde{\delta}(u,v)\nonumber \\
 &  & \hskip-1.5em=\,(u+v+p_{-1}uv)\nonumber \\
 &  & \quad\times\bigg((1-\epsilon u)(1+\epsilon v)Q\!\left(\frac{u}{1+\epsilon u},\frac{v}{1-\epsilon v}\right)-(1+\epsilon u)(1-\epsilon v)Q\!\left(\frac{u}{1-\epsilon u},\frac{v}{1+\epsilon v}\right)\!\bigg)\nonumber \\
 &  & -\,(u+v+q_{-1}uv)\nonumber \\
 &  & \quad\times\bigg((1-\epsilon u)(1+\epsilon v)P\!\left(\frac{u}{1+\epsilon u},\frac{v}{1-\epsilon v}\right)-(1+\epsilon u)(1-\epsilon v)P\!\left(\frac{u}{1-\epsilon u},\frac{v}{1+\epsilon v}\right)\!\bigg)\nonumber \\
 &  & +\,u^{2}v^{2}\Bigg(P\!\left(\frac{u}{1+\epsilon u},\frac{v}{1-\epsilon v}\right)Q\!\left(\frac{u}{1-\epsilon u},\frac{v}{1+\epsilon v}\right)\nonumber \\
 &  & \qquad\qquad-\,P\!\left(\frac{u}{1-\epsilon u},\frac{v}{1+\epsilon v}\right)Q\!\left(\frac{u}{1+\epsilon u},\frac{v}{1-\epsilon v}\right)\!\Bigg).\label{eq:delta_tilde}
\end{eqnarray}

We shall need the following simple observation. Let $f(u)$ be a holomorphic
function in a neighborhood of $0$ (and still assuming that $\epsilon\in\mathbb{C}$,
$\epsilon\neq0$). Then the function 
\[
(1+\epsilon u)f\!\left(\frac{u}{1-\epsilon u}\right)-(1-\epsilon u)f\!\left(\frac{u}{1+\epsilon u}\right)
\]
vanishes identically in a neighborhood of $0$ if and only if the
same is true for $f(u)$.

Indeed, suppose that $f(u)$ does not vanishes identically in a neighborhood
of $0$ and let $n\in\N_{0}$ be the multiplicity of the root of $f(u)$
at $u=0$. Then $f(u)=f_{n}u^{n}+O(u^{n+1})$ as $u\to0$ where $f_{n}\in\mathbb{C}$,
$f_{n}\neq0$, and one immediately finds that
\[
(1+\epsilon u)f\!\left(\frac{u}{1-\epsilon u}\right)-(1-\epsilon u)f\!\left(\frac{u}{1+\epsilon u}\right)=2\epsilon(n+1)f_{n}u^{n+1}+O(u^{n+2}).
\]

From (\ref{eq:delta_tilde}) it is seen that 
\begin{eqnarray}
(1-\epsilon^{2}u^{2})(1-\epsilon^{2}v^{2})\tilde{\delta}(u,v) & = & \Bigg(\!(1+\epsilon v)\left(Q_{0}\!\left(\frac{v}{1-\epsilon v}\right)-P_{0}\!\left(\frac{v}{1-\epsilon v}\right)\right)\nonumber \\
 &  & \quad-\,(1-\epsilon v)\left(Q_{0}\!\left(\frac{v}{1+\epsilon v}\right)-P_{0}\!\left(\frac{v}{1+\epsilon v}\right)\right)\!\Bigg)v\nonumber \\
\noalign{\smallskip} &  & +\,\Phi(u,v)u\label{eq:delta_tilde_u}
\end{eqnarray}
where $\Phi(u,v)$ is an analytic function for $u$, $v$ belonging
to a neighborhood of $0$. Hence, by the above observation, if $P_{0}(y)\neq Q_{0}(y)$
and $v\neq0$ is sufficiently small then the leading coefficient in
the asymptotic expansion of the RHS of (\ref{eq:delta_tilde_u}),
as $u\to0$, is nonzero.

Consider now the case $P_{0}(y)=Q_{0}(y)\neq0$ ($P(x,y)=Q(x,y)\neq0$)
and $p_{-1}\neq q_{-1}$. Then equation (\ref{eq:delta_tilde}) simplifies
to
\begin{eqnarray*}
 &  & \hskip-2em(1-\epsilon^{2}u^{2})(1-\epsilon^{2}v^{2})\tilde{\delta}(u,v)\\
 &  & \hskip-1em=(p_{-1}-q_{-1})uv\\
 &  & \times\bigg(\!(1-\epsilon u)(1+\epsilon v)P\!\left(\frac{u}{1+\epsilon u},\frac{v}{1-\epsilon v}\right)-(1+\epsilon u)(1-\epsilon v)P\!\left(\frac{u}{1-\epsilon u},\frac{v}{1+\epsilon v}\right)\!\bigg).
\end{eqnarray*}
From here it is seen that
\begin{eqnarray}
 &  & \hskip-2em(1-\epsilon^{2}u^{2})(1-\epsilon^{2}v^{2})\tilde{\delta}(u,v)\nonumber \\
 &  & =\,(p_{-1}-q_{-1})\bigg((1+\epsilon v)P_{0}\left(\frac{v}{1-\epsilon v}\right)-(1-\epsilon v)P_{0}\left(\frac{v}{1+\epsilon v}\right)\!\bigg)vu+\Psi(u,v)u^{2}\nonumber \\
\label{eq:delta_tilde0_u}
\end{eqnarray}
where $\Psi(u,v)$ is an analytic function for $u$, $v$ belonging
to a neighborhood of $0$.

Referring again to the above observation we conclude that the leading
coefficient in the asymptotic expansion of the RHS of (\ref{eq:delta_tilde0_u}),
as $u\to0$, is nonzero for all sufficiently small $v\neq0$. \end{proof}

\section{Proof of the main theorem}

\label{sec:proofs}

\subsection{General equations}

Consider semi-infinite matrices $\mathcal{J}$ and $\mathcal{H}$
indexed by $m,n\in\N_{0}$ where $\mathcal{J}$ is a Jacobi matrix~\eqref{eq:def_jacobi}
determined by two complex sequences $\{\alpha_{n}\}$ and $\{\beta_{n}\}$,
with $\alpha_{n}\neq0$ for all $n\geq0$, and $\mathcal{H}$ is a
Hankel matrix~\eqref{eq:def_hankel}. It is convenient and common
to set $\alpha_{-1}:=0$.

Matrices $\mathcal{H}$ and $\mathcal{J}$ commute if and only if
it holds true that
\begin{equation}
(\alpha_{n}-\alpha_{m})h_{n+m+1}+(\beta_{n}-\beta_{m})h_{n+m}+(\alpha_{n-1}-\alpha_{m-1})h_{n+m-1}=0,\label{eq:a_b_H}
\end{equation}
for all $m,n\in\N_{0}$ ($h_{-1}$ is arbitrary). In particular, letting
$m=0$ we have
\begin{equation}
(\alpha_{n}-\alpha_{0})h_{n+1}+(\beta_{n}-\beta_{0})h_{n}+\alpha_{n-1}h_{n-1}=0,\label{eq:H_descend}
\end{equation}
for all $n\in\N$. Taking into account the descending recurrence it
is clear that, for any $n\in\N_{0}$, if $h_{n}=h_{n+1}=0$, then
$h_{0}=h_{1}=\ldots=h_{n}=h_{n+1}=0$.

\subsection{Wilson}

We shall need the asymptotic expansions of $\alpha_{n}^{\text{W}}$
and $\beta_{n}^{\text{W}}$,
\begin{equation}
\alpha_{n}^{\text{W}}=\frac{n^{2}}{4}+\frac{sn}{4}+A_{0}+O\!\left(\frac{1}{n}\right)\!,\quad\beta_{n}^{\text{W}}=\frac{n^{2}}{2}+\frac{(s-1)n}{2}+B_{0}+O\!\left(\frac{1}{n}\right)\text{ }\ \text{as}\ n\to\infty,\label{eq:alp_bet_asympt_wilson}
\end{equation}
where $s:=a+b+c+d$ and $A_{0}$, $B_{0}$ are constants. Explicitly,
\begin{align*}
A_{0}= & \frac{1}{32}\,(-3+4a+4b+4c+4d-2a^{2}-2b^{2}-2c^{2}-2d^{2}+4ab+4ac+4ad+4bc\\
 & \hskip1.5em+4bd+4cd),\\
B_{0}= & \frac{1}{8}\,(-a^{2}-b^{2}-c^{2}-d^{2}+2ab+2ac+2ad+2bc+2bd+2cd).
\end{align*}

The following proposition is in fact an implication stated in Theorem~\ref{thm:main}
ad~(i). Its proof is more straightforward than that for the opposite
direction.

\begin{prop}\label{prop:polyn_impl_nontriv_wilson} If $\alpha_{n}^{\text{W}}$
and $\beta_{n}^{\text{W}}$ depend both polynomially on~$n$, then
$\Ch(\mathcal{J}^{\text{W}})$ is nontrivial if and only if $\alpha_{-1}^{\text{W}}=0$.
If so, $\Ch(\mathcal{J}^{\text{W}})=\spn\left(\mathcal{H}^{(1)},\mathcal{H}^{(2)}\right)$,
where $\mathcal{H}^{(1)}$ and $\mathcal{H}^{(2)}$ are Hankel matrices
determined by the sequences
\begin{equation}
h_{k}^{(1)}=\frac{(-1)^{k}}{k+s-1}\quad\mbox{ and }\quad h_{k}^{(2)}=(-1)^{k},\quad k\in\N_{0},\label{eq:sol_h_wilson}
\end{equation}
respectively. \end{prop}

\begin{proof} Suppose $\alpha_{n}^{\text{W}}$ and $\beta_{n}^{\text{W}}$
are polynomials in~$n$. Then the form of the polynomials is seen
from the asymptotic formulas~\eqref{eq:alp_bet_asympt_wilson} in
which the Landau symbol becomes zero. Plugging these expressions into~\eqref{eq:a_b_H}
and canceling the common term $(n-m)$ one arrives at a recurrence
equation for the entries of the Hankel matrix~$\mathcal{H}$ in which
the indices occur in the combination $m+n$ only. Writing $k$ instead
of $m+n$, the recurrence reads
\[
(k+s)h_{k+1}+2(k+s-1)h_{k}+(k+s-2)h_{k-1}=0,\quad k\in\N.
\]
Its two linearly independent solutions are given by~\eqref{eq:sol_h_wilson}.
Consequently, any $\mathcal{H}\in\Ch(\mathcal{J}^{\text{W}})$ has
to be a linear combination of $\mathcal{H}^{(1)}$ and $\mathcal{H}^{(2)}$.
Moreover, it is easy to see that any such~$\mathcal{H}$, if nontrivial,
can commute with~$\mathcal{J}^{\text{W}}$ if and only if $\alpha_{-1}^{\text{W}}=0$.
\end{proof}

In order to complete the proof of Theorem~\ref{thm:main} ad (i)
we have to prove the implication opposite to that stated in Proposition
\ref{prop:polyn_impl_nontriv_wilson}. To this end, we need the following
equivalent condition for a polynomial dependence of $\alpha_{n}^{\text{W}}$
and $\beta_{n}^{\text{W}}$ on~$n$ to hold true.

\begin{lem}\label{lem:alp_bet_polyn_wilson} The following two statements
are equivalent:
\begin{enumerate}
\item $\alpha_{n}^{\text{W}}$ and $\beta_{n}^{\text{W}}$ depend both polynomially
on~$n\in\mathbb{N}_{0}$.
\item There exists a constant $\omega$ such that
\begin{equation}
\beta_{n}^{\text{W}}-\alpha_{n}^{\text{W}}-\alpha_{n-1}^{\text{W}}+\omega=0,\quad\forall n\in\mathbb{N}.\label{eq:beta_alpha_alpha_wilson}
\end{equation}
\end{enumerate}
Moreover, statement (2) can be true only if $\omega=1/16$. \end{lem}

\begin{proof} If $\alpha_{n}^{\text{W}}$ and $\beta_{n}^{\text{W}}$
depend both polynomially on~$n$ then the form of the polynomials
can be deduced from the asymptotic formulas~\eqref{eq:alp_bet_asympt_wilson}
in which the Landau symbols vanish identically. The implication (1)$\Rightarrow$(2)
is then a matter of a straightforward calculation showing that (\ref{eq:beta_alpha_alpha_wilson})
holds provided we let $\omega=1/16$.

The opposite implication (2)$\Rightarrow$(1) is more tedious. From
\eqref{eq:alp_bet_asympt_wilson} one finds that
\[
\beta_{n}^{\text{W}}-\alpha_{n}^{\text{W}}-\alpha_{n-1}^{\text{W}}+\omega=\omega-\frac{1}{16}\,+O\!\left(\frac{1}{n}\right)\text{ }\ \text{as}\ n\to\infty.
\]
Hence necessarily $\omega=1/16$.

Further we deduce from \eqref{eq:beta_alpha_alpha_wilson} some necessary
conditions on the parameters $a,b,c,d$. The equation implies that
\begin{equation}
\left(\left(\beta_{n}^{\text{W}}+\omega\right)^{2}-(\alpha_{n}^{\text{W}})^{2}-(\alpha_{n-1}^{\text{W}})^{2}\right)^{\!2}-4(\alpha_{n}^{\text{W}})^{2}(\alpha_{n-1}^{\text{W}})^{2}=0,\quad\forall n\in\N.\label{eq:A_C_id_wilson_inproof}
\end{equation}
We temporarily denote by $f(n)$ the left-hand side of~\eqref{eq:A_C_id_wilson_inproof}.
From \eqref{eq:def_alpha_beta_wilson}, \eqref{eq:def_A_wilson} and~\eqref{eq:def_C_wilson}
it is seen that $f(n)$ is a rational function of~$n$. Since $f(n)$
has infinitely many roots it vanishes identically as a complex function.
This is equivalent to saying that the rational function $f(n)$ has
no poles and vanishes at infinity.

Let us check possible poles of $f(n)$. From (\ref{eq:A_C_id_wilson_inproof})
and \eqref{eq:def_alpha_beta_wilson}, \eqref{eq:def_A_wilson}, \eqref{eq:def_C_wilson}
it is clear that the poles can occur only at the points $z_{j}=(4-j-s)/2$,
$j=1,2,3,4,5$. Here we still denote $s=a+b+c+d$. As a necessary
condition, we require that, looking at the Laurent expansion of $f(n)$
around each isolated singularity $z_{j}$, the coefficient corresponding
to the highest possible order of the pole that can occur in the expansion
vanishes. Doing so we shall confine ourselves only to points $z_{1}$
and $z_{2}$, however, since it turns out that inspection of the remaining
points does not lead to additional constraints.

As far as the singular point $z_{1}=(3-s)/2$ is concerned, this singularity
may occur only in the term $(\alpha_{n-1}^{\text{W}})^{2}$, and it
does not occur in $\beta_{n}^{\text{W}}$ and $(\alpha_{n}^{\text{W}})^{2}$.
By inspection of the left-hand side of~\eqref{eq:A_C_id_wilson_inproof}
one readily obtains a condition on the residue
\[
\Res_{n=z_{1}}(\alpha_{n-1}^{\text{W}})^{2}=0,\ \text{equivalently,\ }\Res_{n=z_{3}}(\alpha_{n}^{\text{W}})^{2}=0.
\]
Using the explicit definition of $(\alpha_{n}^{\text{W}})^{2}$ it
is straightforward to derive that the residue is zero if and only
if
\begin{equation}
\left((a+b-c-d)^{2}-1\right)\left((a-b+c-d)^{2}-1\right)\left((a-b-c+d)^{2}-1\right)(s-3)(s-1)=0.\label{eq:1st_cond_wilson_inproof}
\end{equation}

As for the singular point $z_{2}=(2-s)/2$, this pole may occurred
in $\beta_{n}^{\text{W}}$ with order $1$ and in $(\alpha_{n-1}^{\text{W}})^{2}$
with order $2$, but not in $(\alpha_{n}^{\text{W}})^{2}$. An inspection
of the left-hand side of~\eqref{eq:A_C_id_wilson_inproof} shows
that the coefficient standing at $(n-z_{2})^{-2}$ in the Laurent
expansion around $z_{2}$ of $(\beta_{n}^{\text{W}})^{2}-(\alpha_{n-1}^{\text{W}})^{2}$
must be zero. As one can check, $\Res_{n=z_{2}}A_{n-1}^{\text{W}}=-\Res_{n=z_{2}}C_{n}^{\text{W}}$
and therefore this condition reduces to
\[
\Res_{n=z_{2}}C_{n}^{\text{W}}=0.
\]
This is true if and only if
\begin{equation}
(a+b-c-d)(a-b+c-d)(a-b-c+d)(s-2)=0.\label{eq:2st_cond_wilson_inproof}
\end{equation}

Equations~\eqref{eq:1st_cond_wilson_inproof} and~\eqref{eq:2st_cond_wilson_inproof}
imply a finite number of constrains on the parameters $a,b,c,d$ which
can be discussed case by case. Given that $\alpha_{n}^{\text{W}}$
and $\beta_{n}^{\text{W}}$ are symmetric in $a,b,c,d$ it suffices
to consider the following four cases covering all possibilities up
to a permutation of the parameters:

(i)~$a+b+c+d=1$, $a+b-c-d=0$,

(ii)~$a+b+c+d=3$, $a+b-c-d=0$,

(iii)~$a+b-c-d=1$, $a+b+c+d=2$,

(iv)~$a+b-c-d=1$, $a-b+c-d=0$.\\
 In order to complete the discussion successfully we will also take
into account the asymptotic behavior of $\beta_{n}^{\text{W}}-\alpha_{n}^{\text{W}}-\alpha_{n-1}^{\text{W}}+1/16$
at infinity.

In case (i) we have $b=1/2-a$, $d=1/2-c$. Then a computation shows
that
\[
\beta_{n}^{\text{W}}-\alpha_{n}^{\text{W}}-\alpha_{n-1}^{\text{W}}+\frac{1}{16}=\frac{(4a-1)^{2}(4c-1)^{2}}{256}\,\frac{1}{n^{2}}+O\!\left(\frac{1}{n^{3}}\right)\ \text{as}\ n\to\infty.
\]
Hence we have to put either $a=1/4$ or $c=1/4$ but the two choices
differ just by a permutation of parameters. We obtain
\begin{align*}
\alpha_{n}^{\text{W}}\!\left(a,\frac{1}{2}-a,\frac{1}{4},\frac{1}{4}\right) & =\frac{1}{64}\,(4\,n+4\,a+1)(4\,n-4\,a+3),\\
\beta_{n}^{\text{W}}\!\left(a,\frac{1}{2}-a,\frac{1}{4},\frac{1}{4}\right) & =\frac{1}{32}\,(16n^{2}-16a^{2}+8a+1).
\end{align*}

In case (ii) we have $b=3/2-a$, $d=3/2-c$. Then a computation shows
that
\[
\beta_{n}^{\text{W}}-\alpha_{n}^{\text{W}}-\alpha_{n-1}^{\text{W}}+\frac{1}{16}=\frac{(4a-3)^{2}(4c-3)^{2}}{256}\,\frac{1}{n^{2}}+O\!\left(\frac{1}{n^{3}}\right)\ \text{as}\ n\to\infty.
\]
Hence we have to put either $a=3/4$ or $c=3/4$ but again the two
choices differ by a permutation of parameters. We obtain
\begin{align*}
\alpha_{n}^{\text{W}}\!\left(a,\frac{3}{2}-a,\frac{3}{4},\frac{3}{4}\right) & =\frac{1}{64}\,(4\,n+4\,a+3)(4\,n-4\,a+9),\\
\beta_{n}^{\text{W}}\!\left(a,\frac{3}{2}-a,\frac{3}{4},\frac{3}{4}\right) & =\frac{1}{32}\,(16n^{2}+32n-16a^{2}+24a+9).
\end{align*}

In case (iii) we have $b=3/2-a$, $d=1/2-c$. Then a computation shows
that 
\[
\beta_{n}^{\text{W}}-\alpha_{n}^{\text{W}}-\alpha_{n-1}^{\text{W}}+\frac{1}{16}=\frac{(4a-3)^{2}(4c-1)^{2}}{256}\,\frac{1}{n^{2}}+O\!\left(\frac{1}{n^{3}}\right)\ \text{as}\ n\to\infty.
\]
Now we have to distinguish two sub-cases. In sub-case (iiia) we let
$a=3/4$, $b=3/4$, $d=1/2-c$, and we obtain
\begin{align*}
\alpha_{n}^{\text{W}}\!\left(\frac{3}{4},\frac{3}{4},c,\frac{1}{2}-c\right) & =\frac{1}{64}\,(4n-4c+5)(4n+4c+3),\\
\beta_{n}^{\text{W}}\!\left(\frac{3}{4},\frac{3}{4},c,\frac{1}{2}-c\right) & =\frac{1}{32}\,(16n^{2}+16n-16c^{2}+8c+5).
\end{align*}
In sub-case (iiib) we let $b=3/2-a$, $c=1/4$, $d=1/4$, and we obtain
\begin{align*}
\alpha_{n}^{\text{W}}\!\left(a,\frac{3}{2}-a,\frac{1}{4},\frac{1}{4}\right) & =\frac{1}{64}\,(4n-4a+7)(4n+4a+1),\\
\beta_{n}^{\text{W}}\!\left(a,\frac{3}{2}-a,\frac{1}{4},\frac{1}{4}\right) & =\frac{1}{32}\,(16n^{2}+16n-16a^{2}+24a-3).
\end{align*}

In case (iv) we have $d=a-1/2$, $c=b-1/2$. Then a computation shows
that
\[
\beta_{n}^{\text{W}}-\alpha_{n}^{\text{W}}-\alpha_{n-1}^{\text{W}}+\frac{1}{16}=\frac{(4a-3)^{2}(4b-3)^{2}}{256}\,\frac{1}{n^{2}}+O\!\left(\frac{1}{n^{3}}\right)\ \text{as}\ n\to\infty.
\]
Hence we have to put either $a=3/4$ or $b=3/4$ but again the two
choices differ by a permutation of parameters. We obtain
\begin{align*}
\alpha_{n}^{\text{W}}\!\left(\frac{3}{4},b,b-\frac{1}{2},\frac{1}{4}\right) & =\frac{1}{8}\,(n+1)(2n+4b-1),\\
\beta_{n}^{\text{W}}\!\left(\frac{3}{4},b,b-\frac{1}{2},\frac{1}{4}\right) & =\frac{1}{16}\big(8n^{2}+4(4b-1)n+8b-3\big).
\end{align*}

We conclude that the four discussed cases show that whenever (\ref{eq:beta_alpha_alpha_wilson})
holds then $\alpha_{n}^{\text{W}}$ and $\beta_{n}^{\text{W}}$ are
polynomials in $n$. \end{proof}

\begin{rem}\label{eq:alp_bet_polyn_param_wilson} As a matter of
fact, in the second part of the proof of Lemma~\ref{lem:alp_bet_polyn_wilson}
we have found all configurations of the parameters, up to a permutation,
when $\alpha_{n}^{\text{W}}$ and $\beta_{n}^{\text{W}}$ are polynomials
in $n$. But not in all found cases the requirement is met that the
polynomial for $\alpha_{n}^{\text{W}}$ vanishes at $n=-1$. Imposing
in addition this requirement we are left with the following admissible
configurations of the parameters $(a,b,c,d)$, again up to a permutation:

(i)~$(-1/4,3/4,1/4,1/4,)$,

(ii)~$(1/4,5/4,3/4,3/4)$,

(iii)~$(3/4,3/4,1/4,1/4)$,

(iv)~$(3/4,b,b-1/2,1/4)$.\\
 But (i), (ii) and (iii) are particular cases of (iv), possibly up
to a permutation: (i) for $b=1/4$, (ii) for $b=5/4$, (iii) for $b=3/4$.
Moreover, taking the constrains on the parameters $a,b,c,d$ given
below equation~\eqref{eq:def_C_wilson} into account, case~(iv)
is admissible only if $b>1/4$. By writing $b=1/4+t/2$, for $t>0$,
we obtain the parametrization used in claim~(i) of Theorem~\ref{thm:main}.
\end{rem}

Now we are ready to prove the opposite implication of Theorem~\ref{thm:main}.

\begin{prop}\label{prop:nontriv_impl_polyn_wilson} If $\Ch(\mathcal{J}^{\text{W}})$
is nontrivial, then $\alpha_{n}^{\text{W}}$ and $\beta_{n}^{\text{W}}$
depend polynomially on~$n$ and $\alpha_{-1}^{\text{W}}=0$. \end{prop}

\begin{proof} We shall proceed by contradiction. Remember, however,
as observed in Proposition~\ref{prop:polyn_impl_nontriv_wilson},
if $\alpha_{n}^{\text{W}}$ and $\beta_{n}^{\text{W}}$ depend both
polynomially on~$n$ but $\alpha_{-1}^{\text{W}}\neq0$ then $\Ch(\mathcal{J}^{\text{W}})$
is trivial. Hence to get a contradiction we assume that $\Ch(\mathcal{J}^{\text{W}})$
is nontrivial and either $\alpha_{n}^{\text{W}}$ or $\beta_{n}^{\text{W}}$
is not a polynomial in $n$.

Furthermore, without loss of generality we can assume that entries
$h_{n}$ of a nonzero Hankel matrix $\mathcal{H}$ commuting with
$\mathcal{J}^{\text{W}}$ are real. We shall also make use of the
fact that $\alpha_{n}^{\text{W}}$ and $\beta_{n}^{\text{W}}$ may
be both regarded as analytic functions in~$n$ for $n$ sufficiently
large.

Along with (\ref{eq:a_b_H}) we consider the equation
\begin{equation}
(\alpha_{n-1}-\alpha_{m+1})h_{n+m+1}+(\beta_{n-1}-\beta_{m+1})h_{n+m}+(\alpha_{n-2}-\alpha_{m})h_{n+m-1}=0.\label{eq:Hprime}
\end{equation}
Let again $s:=a+b+c+d$. From the asymptotic behavior
\[
\alpha_{n}^{\text{W}}=\frac{n^{2}}{4}+\frac{sn}{4}+O(1),\;\beta_{n}^{\text{W}}=\frac{n^{2}}{2}+\frac{(s-1)n}{2}+O(1),\;\alpha_{n-1}^{\text{W}}=\frac{n^{2}}{4}+\frac{(s-2)n}{4}+O(1),
\]
as $n\to\infty$, it is clear that none of the sets of functions in
$n$, $\{\alpha_{n}^{\text{W}},\beta_{n}^{\text{W}},1\}$ or $\{\beta_{n}^{\text{W}},\alpha_{n-1}^{\text{W}},1\}$
or $\{\alpha_{n}^{\text{W}},\alpha_{n-1}^{\text{W}},1\}$, is linearly
dependent.

Let
\begin{align}
\delta_{1}(n,m) & :=\det\!\left(\begin{array}{cc}
\beta_{n}^{\text{W}}-\beta_{m}^{\text{W}} & \alpha_{n-1}^{\text{W}}-\alpha_{m-1}^{\text{W}}\\
\beta_{n-1}^{\text{W}}-\beta_{m+1}^{\text{W}} & \alpha_{n-2}^{\text{W}}-\alpha_{m}^{\text{W}}
\end{array}\right)\!,\nonumber \\
\delta_{2}(n,m) & :=-\det\!\left(\begin{array}{cc}
\alpha_{n}^{\text{W}}-\alpha_{m}^{\text{W}} & \alpha_{n-1}^{\text{W}}-\alpha_{m-1}^{\text{W}}\\
\alpha_{n-1}^{\text{W}}-\alpha_{m+1}^{\text{W}} & \alpha_{n-2}^{\text{W}}-\alpha_{m}^{\text{W}}
\end{array}\right)\!,\label{eq:def_delta123}\\
\delta_{3}(n,m) & :=\det\!\left(\begin{array}{cc}
\alpha_{n}^{\text{W}}-\alpha_{m}^{\text{W}} & \beta_{n}^{\text{W}}-\beta_{m}^{\text{W}}\\
\alpha_{n-1}^{\text{W}}-\alpha_{m+1}^{\text{W}} & \beta_{n-1}^{\text{W}}-\beta_{m+1}^{\text{W}}
\end{array}\right)\!.\nonumber 
\end{align}
According to Lemma~\ref{lem:Mzw}, for all $m$ sufficiently large
there exists $R_{m}\in\mathbb{N}$ such that for all $n\geq R_{m}$,
$\delta_{1}(n,m)\neq0$ and $\delta_{3}(n,m)\neq0$. Then, by equations
(\ref{eq:a_b_H}) and (\ref{eq:Hprime}), the vectors
\begin{equation}
(h_{n+m+1},h_{n+m},h_{n+m-1})\ \ \text{and}\ \ \big(\delta_{1}(n,m),\delta_{2}(n,m),\delta_{3}(n,m)\big)\label{eq:vecs_h_delta}
\end{equation}
are linearly dependent.

We can assume that $R_{m}$ is so large that $\alpha_{n}^{\text{W}}-\alpha_{m}^{\text{W}}>0$
for all $n\geq R_{m}$. Referring to (\ref{eq:a_b_H}) and (\ref{eq:H_descend}),
one can see that the former of the two vectors in (\ref{eq:vecs_h_delta})
is necessarily nonzero, too. Otherwise $h_{n}=0$ identically. Hence
$\delta_{3}(n,m)\neq0$ implies that $h_{n+m-1}\neq0$ for $n\geq R_{m}$.
This in turn implies that $\delta_{2}(n,m)\neq0$ for $n\geq R_{m}$.
Note that, consequently, $\alpha_{n}^{\text{W}}$ cannot be a polynomial
since otherwise $\delta_{2}(n,m)=0$ identically.

Fix sufficiently large $m\in\N_{0}$. Then for all $n\in\N_{0}$,
$n\geq m_{0}:=R_{m}+m$, it we have
\[
h_{n+1}=\psi(n)h_{n},\ \,\text{with}\ \text{ }\psi(n):=\frac{\delta_{1}(n-m,m)}{\delta_{2}(n-m,m)}\,.
\]
It is of importance that $\psi(n)$ can be regarded as a meromorphic
function of $n$ in a neighborhood of $\infty$. Particularly, $\psi(n)$
has an asymptotic expansion to all orders as $n\to\infty$.

In view of Lemma~\ref{lem:Mzw}, there are only three possible types
of asymptotic behavior of $\psi(n)$ as $n\to\infty$:
\begin{eqnarray*}
(\text{I}) &  & \psi(n)=\lambda_{1}\!\left(1+O\!\left(\frac{1}{n}\right)\!\right)\!,\\
(\text{II}) &  & \psi(n)=\lambda_{2}\,n\left(1+O\!\left(\frac{1}{n}\right)\!\right)\!,\\
(\text{III}) &  & \psi(n)=\frac{\lambda_{3}}{n}\!\left(1+O\!\left(\frac{1}{n}\right)\!\right)\!,
\end{eqnarray*}
where $\lambda_{j}\neq0$ for $j=1,2,3$. From here one can deduce
the asymptotic behavior of
\[
h_{n}=h_{m_{0}}\,\prod_{k=m_{0}}^{n-1}\psi(k).
\]
In case (I) we have
\[
h_{n}=c_{1}\lambda_{1}^{\,n}\,n^{\sigma_{1}}\left(1+O\!\left(\frac{1}{n}\right)\!\right)\text{ }\text{as}\ n\to\infty,
\]
for some $c_{1},\sigma_{1}\in\mathbb{R}$, $c_{1}\neq0$. In case
(II) we have
\[
h_{n}=c_{2}\lambda_{2}^{\,n}\,n!\,n^{\sigma_{2}}\left(1+O\!\left(\frac{1}{n}\right)\!\right)\text{ }\text{as}\ n\to\infty,
\]
for some $c_{2},\sigma_{2}\in\mathbb{R}$, $c_{2}\neq0$. In case
(III) we have
\[
h_{n}=\frac{c_{3}\lambda_{3}^{\,n}\,n^{\sigma_{3}}}{n!}\left(1+O\!\left(\frac{1}{n}\right)\!\right)\text{ }\text{as}\ n\to\infty,
\]
for some $c_{3},\sigma_{3}\in\mathbb{R}$, $c_{3}\neq0$.

Rewriting (\ref{eq:a_b_H}) and taking into the account the asymptotic
behavior of $\alpha_{n}^{\text{W}}$ and $\beta_{n}^{\text{W}}$,
we obtain
\begin{align}
 & h_{n+1}+\frac{\beta_{n-m}^{\text{W}}-\beta_{m}^{\text{W}}}{\alpha_{n-m}^{\text{W}}-\alpha_{m}^{\text{W}}}h_{n}+\frac{\alpha_{n-m-1}^{\text{W}}-\alpha_{m-1}^{\text{W}}}{\alpha_{n-m}^{\text{W}}-\alpha_{m}^{\text{W}}}h_{n-1}\nonumber \\
 & =h_{n+1}+2\!\left(1+O\!\left(\frac{1}{n}\right)\!\right)\!h_{n}+\left(1+O\!\left(\frac{1}{n}\right)\!\right)\!h_{n-1}=0.\label{eq:H_eq_asympt}
\end{align}

It is readily seen that the asymptotic behavior of $h_{n}$ of type
(II) and (III) is incompatible with (\ref{eq:H_eq_asympt}). Hence
the only admissible asymptotic behavior of $h_{n}$ is that of type
(I). Without loss of generality we can suppose that $c_{1}=1$. Moreover,
from (\ref{eq:H_eq_asympt}) it is also seen that $\lambda_{1}$ should
solve the equation $\lambda_{1}^{\,2}+2\lambda_{1}+1=0$ whence $\lambda_{1}=-1$.
Furthermore, for the sake of simplicity we will drop the index in
$\sigma_{1}$. Thus we obtain
\begin{equation}
h_{n}=(-1)^{n}(n+1)^{\sigma}\varphi(n)\label{eq:H_subst}
\end{equation}
where
\begin{equation}
\varphi(n)=1+O\!\left(\frac{1}{n}\right)\ \ \text{as}\ n\to\infty.\label{eq:phi_asympt}
\end{equation}
As a matter of fact, $\varphi(n)$ has an asymptotic expansion to
all orders as $n\to\infty$.

Plugging (\ref{eq:H_subst}) into (\ref{eq:a_b_H}) we obtain
\begin{eqnarray}
 &  & \hskip-5em(\alpha_{n}^{\text{W}}-\alpha_{m}^{\text{W}})\!\left(1+\frac{1}{n+1}\right)^{\!\sigma}\varphi(n+1)-(\beta_{n}^{\text{W}}-\beta_{m}^{\text{W}})\varphi(n)\nonumber \\
 &  & \hskip8em+\,(\alpha_{n-1}^{\text{W}}-\alpha_{m-1}^{\text{W}})\!\left(1-\frac{1}{n+1}\right)^{\!\sigma}\varphi(n-1)=0.\label{eq:H_phi}
\end{eqnarray}

The asymptotic expansion of the LHS of~(\ref{eq:H_phi}) as $n\to\infty$,
with $m$ being fixed but otherwise arbitrary, while taking into account
(\ref{eq:phi_asympt}) and \eqref{eq:alp_bet_asympt_wilson} and substituting
for $s$, $A_{0}$, $B_{0}$, yields the expression
\[
\beta_{m}^{\text{W}}-\alpha_{m}^{\text{W}}-\alpha_{m-1}^{\text{W}}+\frac{\sigma(\sigma+1)}{4}+\frac{1}{16}+O\!\left(\frac{1}{n}\right)\!.
\]
Hence
\[
\beta_{m}^{\text{W}}-\alpha_{m}^{\text{W}}-\alpha_{m-1}^{\text{W}}+\frac{\sigma(\sigma+1)}{4}+\frac{1}{16}=0,\text{ }\ \text{for}\ \text{all}\ m\in\N_{0}.
\]
Lemma~\ref{lem:alp_bet_polyn_wilson} implies that $\beta_{n}^{\text{W}}$
and $\alpha_{n}^{\text{W}}$ depend both polynomially on~$n$, a
contradiction. \end{proof}

In summary, as far as the Jacobi matrix $\mathcal{J}^{W}$ is concerned,
the equivalence stated in Theorem~\ref{thm:main} ad~(i) means two
implications which are established by Propositions~\ref{prop:polyn_impl_nontriv_wilson}
and~\ref{prop:nontriv_impl_polyn_wilson}. Along with Remark~\ref{eq:alp_bet_polyn_param_wilson}
this concludes the proof of Theorem~\ref{thm:main} ad~(i).

\subsection{Continuous dual Hahn}

The approach used in the case of Wilson polynomials can be also applied
in the case of Continuous dual Hahn polynomials with no essential
difference. Due to the considerably simpler form of the Jacobi parameters~$\alpha_{n}^{\text{CdH}}$
and~$\beta_{n}^{\text{CdH}}$ (but the same order of the asymptotic
behavior for $n\to\infty$), some parts of the proof simplify significantly.
Observe that~$\beta_{n}^{\text{CdH}}$ is a polynomial in~$n$ for
all values of the parameters $a,b,c$. Namely, we have
\begin{equation}
\beta_{n}^{\text{CdH}}=2n^{2}+(2s-1)n+\tilde{s},\label{eq:bet_asympt_cdhahn}
\end{equation}
where we denote $s:=a+b+c$ and $\tilde{s}:=ab+ac+bc$. In addition,
\begin{equation}
\alpha_{n}^{\text{CdH}}=n^{2}+\frac{2s+1}{2}\,n+\frac{4s+4\tilde{s}-1}{8}+O\!\left(\frac{1}{n}\right)\!,\,\ \text{as}\ n\to\infty.\label{eq:alp_asympt_cdhahn}
\end{equation}

\begin{prop}\label{prop:polyn_impl_nontriv_cdhahn} If $\alpha_{n}^{\text{CdH}}$
depends polynomially on~$n$, then $\Ch(\mathcal{J}^{\text{CdH}})$
is nontrivial. More precisely, if that is the case, $\Ch(\mathcal{J}^{\text{CdH}})=\spn\left(\mathcal{H}^{(1)},\mathcal{H}^{(2)}\right)$,
where $\mathcal{H}^{(1)}$ and $\mathcal{H}^{(2)}$ are Hankel matrices
determined by the sequences
\begin{equation}
h_{k}^{(1)}=\frac{(-1)^{k}}{k+s-1/2}\quad\mbox{ and }\quad h_{k}^{(2)}=(-1)^{k},\quad k\in\N_{0},\label{eq:sol_h_cdhahn}
\end{equation}
respectively. \end{prop}

\begin{proof} Note that $\alpha_{n}^{\text{CdH}}$ depends polynomially
on~$n$ if and only if two of the parameters $a,b,c$ are equal to~$1/2$.
In this case, $\alpha_{-1}^{\text{CdH}}=0$ automatically. Assuming
that $\alpha_{n}^{\text{CdH}}$ is a polynomial in~$n$, the Landau
symbol in~\eqref{eq:alp_asympt_cdhahn} is identically vanishing.
Plugging this expression together with~\eqref{eq:bet_asympt_cdhahn}
into~\eqref{eq:a_b_H}, one infers that any $\mathcal{H}\in\Ch(\mathcal{J}^{\text{CdH}})$,
with $\mathcal{H}_{m,n}=h_{m+n}$, if and only if $\{h_{n}\}$ is
a solution of the three-term recurrence
\[
(k+s+1/2)h_{s+1}+(2k+2s-1)h_{k}+(k+s-3/2)h_{k-1}=0,\quad k\in\N.
\]
The two linearly independent solutions of the above recurrence are
given by~\eqref{eq:sol_h_cdhahn} and the proof follows. \end{proof}

\begin{lem}\label{lem:alp_bet_polyn_cdhahn} The following two statements
are equivalent: 
\begin{enumerate}
\item $\alpha_{n}^{\text{CdH}}$ depends polynomially on $n$.
\item There exists a constant $\omega$ such that
\begin{equation}
\beta_{n}^{\text{CdH}}-\alpha_{n}^{\text{CdH}}-\alpha_{n-1}^{\text{CdH}}+\omega=0,\quad\forall n\in\N_{0}.\label{eq:beta_alpha_alpha_cdhahn}
\end{equation}
\end{enumerate}
Moreover, statement (2) can be true only if $\omega=1/4$. \end{lem}

\begin{proof} The implication (1)$\Rightarrow$(2): If $\alpha_{n}$
is a polynomial in $n$ then, up to a permutation of the parameters,
$a=b=1/2$, whence
\[
\alpha_{n}^{\text{CdH}}=(n+1)\!\left(n+\frac{1}{2}+c\right)
\]
and
\[
\beta_{n}^{\text{CdH}}=n\!\left(n-\frac{1}{2}+c\right)+(n+1)\!\left(n+\frac{1}{2}+c\right)-\frac{1}{4}=\alpha_{n-1}+\alpha_{n}-\frac{1}{4}\,.
\]

The implication (2)$\Rightarrow$(1): The equation \eqref{eq:beta_alpha_alpha_cdhahn}
means that the sequence $\{\alpha_{n}^{\text{CdH}}\}$ obeys a two-term
recurrence with a right-hand side,
\[
\alpha_{n}^{\text{CdH}}+\alpha_{n-1}^{\text{CdH}}=\phi_{n},
\]
where $\phi_{n}:=\beta_{n}^{\text{CdH}}+\omega$ is a polynomial in
$n$ of degree~$2$ with the leading coefficient equal to $2$. It
follows that
\[
\alpha_{n+2}^{\text{CdH}}-\alpha_{n+1}^{\text{CdH}}-\alpha_{n}^{\text{CdH}}+\alpha_{n-1}^{\text{CdH}}=\phi_{n+2}-2\phi_{n+1}+\phi_{n}=4.
\]
Hence $\alpha_{n}^{\text{CdH}}$ must be of the form 
\[
\alpha_{n}^{\text{CdH}}=A+Bn+C(-1)^{n}+n^{2}
\]
where $A$, $B$, $C$ are some constants. At the same time, we know
that the square of $\alpha_{n}^{\text{CdH}}$ is a polynomial in $n$
of degree~$4$. Whence
\[
C(A+Bn+n^{2})\!=0,\,\quad\forall n\in\mathbb{N}_{0}.
\]
Necessarily, $C=0$ and so $\alpha_{n}^{\text{CdH}}=A+Bn+n^{2}$.
\end{proof}

\begin{prop}\label{prop:nontriv_impl_polyn_cdhahn} If $\Ch(\mathcal{J}^{\text{CdH}})$
is nontrivial, then $\alpha_{n}^{\text{CdH}}$ depends polynomially
on~$n$. \end{prop}

\begin{proof} The proof is based on Lemma~\ref{lem:Mzw} and follows
the same steps as the proof of Proposition~\ref{prop:nontriv_impl_polyn_wilson}.
Therefore we just briefly address several points specific for this
case.

For a contradiction, we suppose that a nonzero Hankel matrix $\mathcal{H}$
commuting with $\mathcal{J}^{\text{CdH}}$ exists but $\alpha_{n}^{\text{CdH}}$
is not a polynomial in~$n$. It follows from~\eqref{eq:alp_asympt_cdhahn}
that 
\[
\alpha_{n-1}^{\text{CdH}}=n^{2}+\frac{2s-3}{2}\,n-\frac{4s-4\tilde{s}-3}{8}+O\!\left(\frac{1}{n}\right),\,\ \text{as}\ n\to\infty.
\]
This together with~\eqref{eq:bet_asympt_cdhahn} and~\eqref{eq:alp_asympt_cdhahn}
implies that none of the sets of functions in $n$, $\{\alpha_{n},\beta_{n},1\}$
or $\{\beta_{n},\alpha_{n-1},1\}$ or $\{\alpha_{n},\alpha_{n-1},1\}$,
is linearly dependent.

Now, using exactly the same reasoning as in the proof of Proposition~\ref{prop:nontriv_impl_polyn_wilson}
while considering $\alpha_{n}^{\text{CdH}}$ and $\beta_{n}^{\text{CdH}}$
instead of $\alpha_{n}^{\text{W}}$ and $\beta_{n}^{\text{W}}$, one
can argue that~\eqref{eq:beta_alpha_alpha_cdhahn} is necessarily
true. By Lemma~\ref{lem:alp_bet_polyn_cdhahn}, this means a contradiction.
\end{proof}

\subsection{Continuous Hahn}

One has, as $n\to\infty$,
\begin{equation}
\alpha_{n}^{\text{CH}}=\frac{n}{4}+\frac{a+b+c+d}{8}+O\!\left(\frac{1}{n}\right)\quad\mbox{ and }\quad\beta_{n}^{\text{CH}}=\frac{\ii(a+b-c-d)}{4}+O\!\left(\frac{1}{n}\right)\!.\label{eq:alp_bet_asympt_chahn}
\end{equation}

\begin{prop}\label{prop:polyn_impl_nontriv_chahn} If $\alpha_{n}^{\text{CH}}$
and $\beta_{n}^{\text{CH}}$ depend both polynomially on~$n$, then
$\Ch(\mathcal{J}^{\text{CH}})$ is nontrivial if and only if $\alpha_{-1}^{\text{CH}}=0$.
If so, $\Ch(\mathcal{J}^{\text{CH}})=\spn\left(\mathcal{H}^{(1)},\mathcal{H}^{(2)}\right)$,
where $\mathcal{H}^{(1)}$ and $\mathcal{H}^{(2)}$ are Hankel matrices
determined by the sequences
\begin{equation}
h_{k}^{(1)}=\sin\frac{k\pi}{2}\quad\mbox{ and }\quad h_{k}^{(2)}=\cos\frac{k\pi}{2},\quad k\in\N_{0},\label{eq:sol_h_chahn}
\end{equation}
respectively. \end{prop}

\begin{proof} The proof is again analogous to the proof of Proposition~\ref{prop:polyn_impl_nontriv_wilson}.
Assuming that~$\alpha_{n}^{\text{CH}}$ and $\beta_{n}^{\text{CH}}$
are polynomials, we can deduce the form of the polynomials from \eqref{eq:alp_bet_asympt_chahn},
and the general equation~\eqref{eq:a_b_H} implies that any $\mathcal{H}\in\Ch(\mathcal{J}^{\text{CH}})$,
with entries $\mathcal{H}_{m,n}=h_{m+n}$, must fulfill the simple
recurrence
\[
h_{k+1}+h_{k-1}=0,\quad k\in\N,
\]
whose two linearly independent solutions are given by~\eqref{eq:sol_h_chahn}.
Now it suffices to notice that such~$\mathcal{H}$ is necessarily
trivial if $\alpha_{-1}^{\text{CH}}\neq0$. On the contrary, if $\alpha_{-1}^{\text{CH}}=0$
then any solution of the above recurrence defines a Hankel matrix
which commutes with~$\mathcal{J}^{\text{CH}}$. \end{proof}

\begin{prop}\label{prop:nontriv_impl_polyn_chahn} If $\Ch(\mathcal{J}^{\text{CH}})$
is nontrivial then $\alpha_{n}^{\text{CH}}$ depends polynomially
on~$n$. \end{prop}

\begin{proof} Let us proceed by contradiction and assume that $\alpha_{n}^{\text{CH}}$
is not a polynomial in $n$ but (\ref{eq:a_b_H}) has a nonzero solution~$h_{n}$.
Without loss of generality we can suppose $h_{n}$ to be real. On
a neighborhood of infinity we can expand
\[
\alpha_{n}^{\text{CH}}=\frac{n}{4}+A_{0}+A(n),\,\ \text{where}\ A(n)=\sum_{k=1}^{\infty}\frac{A_{k}}{n^{k}}\,.
\]
By our assumption, $A(n)$ cannot vanish identically. Keeping notation~\eqref{eq:def_delta123},
with $\alpha_{n}^{\text{CH}}$ and $\beta_{n}^{\text{CH}}$ instead
of $\alpha_{n}^{\text{W}}$ and $\beta_{n}^{\text{W}}$, we have,
for a fixed $m$,
\begin{eqnarray*}
\delta_{2}(n,m) & := & -\det\!\left(\begin{array}{cc}
\alpha_{n}^{\text{CH}}-\alpha_{m}^{\text{CH}} & \alpha_{n-1}^{\text{CH}}-\alpha_{m-1}^{\text{CH}}\\
\alpha_{n-1}^{\text{CH}}-\alpha_{m+1}^{\text{CH}} & \alpha_{n-2}^{\text{CH}}-\alpha_{m}^{\text{CH}}
\end{array}\right)\\
 & = & \frac{A(m+1)-2A(m)+A(m-1)}{4}\,n+O(1),\ \ \text{as}\ n\to\infty.
\end{eqnarray*}
If $A(n)=A_{k}n^{-k}+O(n^{-k-1})$, with $A_{k}\neq0$, then
\[
A(m+1)-2A(m)+A(m-1)=A_{k}k(k+1)m^{-k-2}+O(m^{-k-3}),\,\ \text{as}\ m\to\infty.
\]
Hence for all sufficiently large $m$ there exists
\begin{equation}
\lim_{n\to\infty}\frac{\delta_{2}(n,m)}{n}=:C_{m}\in\mathbb{R\setminus}\{0\}.\label{eq:lim_delta2_chahn}
\end{equation}
In particular, there exists $R_{m}\geq0$ such that for all $n\geq R_{m}$,
$\delta_{2}(n,m)\neq0$.

One can argue, similarly as in the proof of Proposition~\ref{prop:nontriv_impl_polyn_wilson},
that the vectors (\ref{eq:vecs_h_delta}) are linearly dependent.
Moreover, for $m$ sufficiently large and $n\geq R_{m}$ both of them
are nonzero (note that $R_{m}$ can be chosen large enough so that
the coefficient $\alpha_{n}-\alpha_{m}$ occurring in (\ref{eq:a_b_H})
is nonzero for all $n\geq R_{m}$).

It follows that $\beta_{n}^{\text{CH}}$ cannot be a polynomial since
if so, $\beta_{n}^{\text{CH}}$ would be a constant as seen from~\eqref{eq:alp_bet_asympt_chahn},
and then $\delta_{1}(n,m)=\delta_{3}(n,m)=0$ identically. Consequently,
$h_{n}$ would vanish for all sufficiently large and hence for all
$n$. Thus, we can expand $\beta_{n}$ as a function of $n$ near
infinity as
\[
\beta_{n}^{\text{CH}}=B_{0}+B(n),\,\ \text{where}\ B(n)=\sum_{k=1}^{\infty}\frac{B_{k}}{n^{k}}\,,
\]
$B_{0}=\ii(a+b-c-d)/4$ and $B(n)$ does not vanish identically. This
property in particular implies that $\beta_{m+1}^{\text{CH}}-\beta_{m}^{\text{CH}}\neq0$
for all sufficiently large $m$.

Now we can check the asymptotic behavior of $\delta_{1}(n,m)$ and
$\delta_{3}(n,m)$ for $m$ fixed and $n$ large. We have, as $n\to\infty$,
\begin{align*}
\delta_{1}(n,m) & :=\det\!\left(\begin{array}{cc}
\beta_{n}^{\text{CH}}-\beta_{m}^{\text{CH}} & \alpha_{n-1}^{\text{CH}}-\alpha_{m-1}^{\text{CH}}\\
\beta_{n-1}^{\text{CH}}-\beta_{m+1}^{\text{CH}} & \alpha_{n-2}^{\text{CH}}-\alpha_{m}^{\text{CH}}
\end{array}\right)=\,\frac{\beta_{m+1}^{\text{CH}}-\beta_{m}^{\text{CH}}}{4}\,n+O(1),\\
\noalign{\smallskip}\delta_{3}(n,m) & :=\det\!\left(\begin{array}{cc}
\alpha_{n}^{\text{CH}}-\alpha_{m}^{\text{CH}} & \beta_{n}^{\text{CH}}-\beta_{m}^{\text{CH}}\\
\alpha_{n-1}^{\text{CH}}-\alpha_{m+1}^{\text{CH}} & \beta_{n-1}^{\text{CH}}-\beta_{m+1}^{\text{CH}}
\end{array}\right)=\,-\frac{\beta_{m+1}^{\text{CH}}-\beta_{m}^{\text{CH}}}{4}\,n+O(1).
\end{align*}
Fix $m$ sufficiently large. Then
\[
\lim_{n\to\infty}\frac{h_{n+1}}{h_{n-1}}=\lim_{n\to\infty}\frac{h_{n+m+1}}{h_{n+m-1}}=\lim_{n\to\infty}\frac{\delta_{1}(n,m)}{\delta_{3}(n,m)}=-1.
\]
In view of (\ref{eq:lim_delta2_chahn}), we also have
\[
\lim_{n\to\infty}\frac{h_{n+1}}{h_{n}}=\lim_{n\to\infty}\frac{h_{n+m+1}}{h_{n+m}}=\lim_{n\to\infty}\frac{\delta_{1}(n,m)}{\delta_{2}(n,m)}=\frac{\beta_{m+1}^{\text{CH}}-\beta_{m}^{\text{CH}}}{4C_{m}}=:D\in\mathbb{R}\setminus\{0\}.
\]
Then
\[
\lim_{n\to\infty}\frac{h_{n+1}}{h_{n-1}}=\lim_{n\to\infty}\frac{h_{n+1}}{h_{n}}\,\frac{h_{n}}{h_{n-1}}=D^{2}.
\]
Hence $D^{2}=-1$, a contradiction. \end{proof}

The next lemma summarizes configurations of the parameters when $\alpha_{n}^{\text{CH}}$
is a polynomial in~$n$.

\renewcommand{\theenumi}{\Roman{enumi}}%

\begin{lem}\label{lem:polyn_param_chahn} $\alpha_{n}^{\text{CH}}$
depends polynomially on~$n$ for the following values of parameters
only:
\begin{enumerate}
\item ${\displaystyle a=b=\frac{1}{4}+\ii t\;\mbox{ and }\;c=d=\frac{1}{4}-\ii t}$,~for
$t\in\R$, in which case
\[
\alpha_{n}^{\text{CH}}=\frac{1}{4}\left(n+\frac{1}{2}\right).
\]
\item ${\displaystyle a=b=\frac{3}{4}+\ii t\;\mbox{ and }\;c=d=\frac{3}{4}-\ii t}$,~for
$t\in\R$, in which case
\[
\alpha_{n}^{\text{CH}}=\frac{1}{4}\left(n+\frac{3}{2}\right).
\]
\item ${\displaystyle a=\frac{1}{4}+\ii t,\,b=\frac{3}{4}+\ii t\;\mbox{ and }\;c=\frac{1}{4}-\ii t,\,d=\frac{3}{4}-\ii t}$,~for
$t\in\R$, or one can interchange $a$ with $b$ and $c$ with $d$.
In both cases,
\[
\alpha_{n}^{\text{CH}}=\frac{1}{4}\left(n+1\right).
\]
\end{enumerate}
Moreover, $\beta_{n}^{\text{CH}}=\ii(a+b-c-d)/4$ in each case~(I)-(III).
\end{lem}

\renewcommand{\theenumi}{\arabic{enumi}}%

\begin{rem}\label{rem:alp-1_chahn} Note that $\alpha_{-1}^{\text{CH}}=0$
only in case~(III). \end{rem}

\begin{proof} Bearing in mind the restriction~\eqref{eq:param_assum_chahn},
we temporarily reparameterize the coefficients as follows:
\[
\alpha:=\Re a,\quad\beta:=\Re b,\quad\phi:=\Im a-\Im b,\quad\theta:=\Im a+\Im b,
\]
where $\alpha,\beta>0$ and $\phi,\theta\in\mathbb{R}$. As a shortcut
we again put $s:=a+b+c+d=2\alpha+2\beta$. We have
\begin{equation}
\left(\alpha_{n}^{\text{CH}}\right)^{2}=\frac{(n+1)(n+2\alpha)(n+2\beta)(n+s-1)(2n+s+2\ii\phi)(2n+s-2\ii\phi)}{4\,(2n+s-1)(2n+s)^{2}(2n+s+1)}.\label{eq:alpha_n_2_chahn}
\end{equation}

If $\alpha_{n}^{\text{CH}}$ is a polynomial in $n$, then it is of
degree $1$ with the leading coefficient equal to $1/4$ as is obvious
from~\eqref{eq:alp_bet_asympt_chahn}. In that case, it is immediately
seen from~\eqref{eq:alpha_n_2_chahn} that $\phi=0$ and hence
\[
\left(\alpha_{n}^{\text{CH}}\right)^{2}=\frac{(n+1)(n+2\alpha)(n+2\beta)(n+s-1)}{4\,(2n+s-1)(2n+s+1)}.
\]
Then by an elementary inspection of possible cancellation and requiring
the above expression to be the square of a polynomial in~$n$, one
readily finds that there are only four possible configurations for
$\alpha$ and $\beta$. Namely, $\alpha=\beta=1/4$, $\alpha=\beta=3/4$,
$\alpha=1/4$ and $\beta=3/4$, or $\alpha=3/4$ and $\beta=1/4$.
Moreover, in each case, $\beta_{n}^{\text{CH}}=-\theta/2$. The lemma
follows. \end{proof}

In summary, one implication of the equivalence in Theorem~\ref{thm:main}
ad~(iii) is established by Proposition~\ref{prop:polyn_impl_nontriv_chahn}.
Conversely, if $\Ch(\mathcal{J}^{\text{CH}})$ is nontrivial then
$\alpha_{n}^{\text{CH}}$ depends polynomially on~$n$, by Proposition~\ref{prop:nontriv_impl_polyn_chahn}.
In this case, $\beta_{n}^{\text{CH}}$ is a polynomial in~$n$, too,
as it follows from Lemma~\ref{lem:polyn_param_chahn}. The proof
of Theorem~\ref{thm:main} ad~(iii) is then completed by Proposition~\ref{prop:polyn_impl_nontriv_chahn},
Lemma~\ref{lem:polyn_param_chahn}, and Remark~\ref{rem:alp-1_chahn}.

\subsection{Jacobi}

The proof of Theorem~\ref{thm:main} ad (iv) asserting that $\Ch(\mathcal{J}^{\text{J}})$
is trivial for any admissible choice of the parameters $\alpha$,
$\beta$ is more elementary than in the foregoing cases but computationally
rather demanding. Despite the fact that the coefficients $\alpha_{n}^{\text{J}}$
and $\beta_{n}^{\text{J}}$ are given by less complicated expressions
derivations of some asymptotic formulas used in the course of the
proof are best done with the aid of a convenient computer algebra
system (CAS). To our opinion, any commonly used CAS will do.

In order to obtain slightly less complicated expressions, we use a
new parametrization: $c:=\alpha+\beta$ and $d:=\beta-\alpha$. The
coefficients~\eqref{eq:def_alpha_jacobi} and~\eqref{eq:def_beta_jacobi}
take the form
\begin{equation}
\alpha_{n}^{\text{J}}=\sqrt{\frac{(n+1)(2n+c+d+2)(2n+c-d+2)(n+c+1)}{(2n+c+1)(2n+c+2)^{2}(2n+c+3)}}\label{eq:def_alpha_jacobi_reparam}
\end{equation}
and
\begin{equation}
\beta_{n}^{\text{J}}=\frac{cd}{(2n+c)(2n+c+2)},\label{eq:def_beta_jacobi_reparam}
\end{equation}
where we can assume that $c>-2$ and $d\geq0$ without loss of generality.
Indeed, since $\alpha_{n}^{\text{J}}(\alpha,\beta)=\alpha_{n}^{\text{J}}(\beta,\alpha)$
and $\beta_{n}^{\text{J}}(\alpha,\beta)=-\beta_{n}^{\text{J}}(\beta,\alpha)$
for all $n\in\N_{0}$, a sequence $h_{n}$ is a solution of~\eqref{eq:a_b_H}
if and only if $(-1)^{n}h_{n}$ is a solution of the same equation
with interchanged parameters $\alpha$ and $\beta$. Therefore we
may assume $\beta\geq\alpha$.

Writing
\[
\alpha_{n}^{\text{J}}=\frac{1}{2}+A(n)
\]
we have
\begin{eqnarray}
A(n) & = & \frac{1-c^{2}-d^{2}}{16\,n^{2}}-\frac{(1-c^{2}-d^{2})(c+2)}{16\,n^{3}}\nonumber \\
\noalign{\smallskip} &  & +\left(\frac{(1-c^{2}-d^{2})(13\,c^{2}+d^{2}+48\,c+51)}{256}+\frac{c^{2}d^{2}}{64}\right)\!\frac{1}{n^{4}}+O\!\left(\frac{1}{n^{5}}\right)\!,\,\ \text{as}\ n\to\infty.\nonumber \\
\label{eq:AJasympt}
\end{eqnarray}
Furthermore,
\begin{equation}
\beta_{n}^{\text{J}}=\frac{cd}{4\,n^{2}}-\frac{cd\,(c+1)}{4\,n^{3}}+\frac{cd\,(3c^{2}+6\,c+4)}{16\,n^{4}}-\frac{cd\,(c^{3}+3\,c^{2}+4\,c+2)}{8\,n^{5}}+O\!\left(\frac{1}{n^{6}}\right)\!,\,\ \text{as}\ n\to\infty.\label{eq:betaJasympt}
\end{equation}

We again assume that a nontrivial real solution~$h_{n}$ to \eqref{eq:a_b_H}
exists, with the coefficients $\alpha_{n}$ and $\beta_{n}$ being
given by \eqref{eq:def_alpha_jacobi_reparam} and \eqref{eq:def_beta_jacobi_reparam}.
Moreover, we keep the notation \eqref{eq:def_delta123} using, however,
the coefficients $\alpha_{n}^{\text{J}}$, $\beta_{n}^{\text{J}}$
rather than $\alpha_{n}^{\text{W}}$, $\beta_{n}^{\text{W}}$. The
asymptotic formulas lead us to distinguish four cases.

\vskip10pt 1)~\textbf{Case $c^{2}+d^{2}=1$, $cd=0$.} Then we have
three possible configurations of the parameters: $c=0$, $d=1$ or
$c=-1$, $d=0$ or $c=1$, $d=0$. In any case $\alpha_{n}^{\text{J}}=1/2$
is a constant sequence. This means that the boundary condition $\alpha_{-1}^{\text{J}}=0$
is not fulfilled and therefore $\Ch(\mathcal{J}^{\text{J}})$ is trivial.

\vskip10pt 2)~\textbf{Case $c^{2}+d^{2}\neq1$, $cd=0$.} Then $\beta_{n}^{\text{J}}=0$
identically and equations \eqref{eq:a_b_H} and \eqref{eq:Hprime}
simplify to
\begin{eqnarray}
 &  & (\alpha_{n}-\alpha_{m})h_{n+m+1}+(\alpha_{n-1}-\alpha_{m-1})h_{n+m-1}=0,\nonumber \\
 &  & (\alpha_{n-1}-\alpha_{m+1})h_{n+m+1}+(\alpha_{n-2}-\alpha_{m})h_{n+m-1}=0.\label{eq:J-a_b_H-b0}
\end{eqnarray}
From \eqref{eq:AJasympt} it is seen that
\begin{equation}
\delta_{2}(n,m)=-\det\!\left(\begin{array}{cc}
\alpha_{n}^{\text{J}}-\alpha_{m}^{\text{J}} & \alpha_{n-1}^{\text{J}}-\alpha_{m-1}^{\text{J}}\\
\alpha_{n-1}^{\text{J}}-\alpha_{m+1}^{\text{J}} & \alpha_{n-2}^{\text{J}}-\alpha_{m}^{\text{J}}
\end{array}\right)=A(m+1)A(m-1)-A^{2}(m)+O\!\left(\frac{1}{n^{2}}\right)\!,\label{eq:delta_2_asympt_jacobi}
\end{equation}
as $n\to\infty$. Moreover,
\begin{equation}
A(m+1)A(m-1)-A^{2}(m)=\frac{(1-c^{2}+d^{2})^{2}}{128\,m^{6}}+O\!\left(\frac{1}{m^{7}}\right)\!,\,\ \text{as}\ m\to\infty.\label{eq:delta_2_lead_m_asympt_jacobi}
\end{equation}

Consequently, whenever \textbf{$c^{2}+d^{2}\neq1$} then for all sufficiently
large $m\in\mathbb{N}$ there exists $R_{m}\in\N$ such that $\delta_{2}(n,m)\neq0$
for all $n\geq R_{m}$. It readily follows from \eqref{eq:J-a_b_H-b0}
that $h_{n}=0$ for all sufficiently large $n$. Then necessarily
$h_{n}=0$ for all $n\in\mathbb{N}_{0}$ (see \eqref{eq:H_descend}),
a contradiction.

\vskip10pt 3)~\textbf{Case $c^{2}+d^{2}\neq1$, $cd\neq0$.} In
this case, too, we have $\delta_{2}(n,m)\neq0$ for $n\geq R_{m}$
provided $m$ is sufficiently large. Moreover, the asymptotic expansions
\eqref{eq:delta_2_asympt_jacobi}, \eqref{eq:delta_2_lead_m_asympt_jacobi}
are valid. In addition, as $n\to\infty$, we have
\begin{equation}
\delta_{1}(n,m)=\det\!\left(\begin{array}{cc}
\beta_{n}^{\text{J}}-\beta_{m}^{\text{J}} & \alpha_{n-1}^{\text{J}}-\alpha_{m-1}^{\text{J}}\\
\beta_{n-1}^{\text{J}}-\beta_{m+1}^{\text{J}} & \alpha_{n-2}^{\text{J}}-\alpha_{m}^{\text{J}}
\end{array}\right)=\beta_{m}^{\text{J}}A(m)-\beta_{m+1}^{\text{J}}A(m-1)+O\!\left(\frac{1}{n^{2}}\right)\label{eq:delta_1_asympt_jacobi}
\end{equation}
and
\begin{equation}
\delta_{3}(n,m)=\det\!\left(\begin{array}{cc}
\alpha_{n}^{\text{J}}-\alpha_{m}^{\text{J}} & \beta_{n}^{\text{J}}-\beta_{m}^{\text{J}}\\
\alpha_{n-1}^{\text{J}}-\alpha_{m+1}^{\text{J}} & \beta_{n-1}^{\text{J}}-\beta_{m+1}^{\text{J}}
\end{array}\right)=\beta_{m+1}^{\text{J}}A(m)-\beta_{m}^{\text{J}}A(m+1)+O\!\left(\frac{1}{n^{2}}\right)\!.\label{eq:delta_3_asympt_jacobi}
\end{equation}
Using (\ref{eq:AJasympt}), (\ref{eq:betaJasympt}) one derives that
\begin{equation}
\beta_{m}^{\text{J}}A(m)-\beta_{m+1}^{\text{J}}A(m-1)=\frac{cd\,(c^{2}+d^{2}-1)}{64\,m^{6}}+O\!\left(\frac{1}{m^{7}}\right)\!,\,\ \text{as}\ m\to\infty,\label{eq:delta_1_lead_m_asympt_jacobi}
\end{equation}
and
\begin{equation}
\beta_{m+1}^{\text{J}}A(m)-\beta_{m}^{\text{J}}A(m+1)=\frac{cd\,(c^{2}+d^{2}-1)}{64\,m^{6}}+O\!\left(\frac{1}{m^{7}}\right)\!,\,\ \text{as}\ m\to\infty.\label{eq:delta_3_lead_m_asympt_jacobi}
\end{equation}
Hence for all sufficiently large $m\in\mathbb{N}$ there exists $R_{m}\in\mathbb{N}$
such that for all $n\geq R_{m}$ and $j=1,2,3$, $\delta_{j}(n,m)\neq0$.

Similarly as in the proof of Proposition~\ref{prop:nontriv_impl_polyn_wilson}
we can infer from here some information about the asymptotic behavior
of $h_{n}$ as $n\to\infty$. We only sketch the basic steps. Fix
sufficiently large $m\in\mathbb{N}$ but otherwise arbitrary. The
vectors~\eqref{eq:vecs_h_delta} are linearly dependent and therefore
for all $n\in\mathbb{N}$, $n\geq m_{0}:=R_{m}+m$, we have
\[
h_{n+1}=\psi(n)h_{n},\ \,\text{with}\ \text{ }\psi(n):=\frac{\delta_{1}(n-m,m)}{\delta_{2}(n-m,m)}\,.
\]
From (\ref{eq:delta_2_asympt_jacobi}) and (\ref{eq:delta_1_asympt_jacobi})
it is seen that $\psi(n)=\lambda\big(1+O(n^{-2})\big)$ for some $\lambda\neq0$.
Whence
\begin{equation}
h_{n}=h_{m_{0}}\,\prod_{k=m_{0}}^{n-1}\psi(k)=\gamma\,\lambda^{\,n}\!\left(1+O\!\left(\frac{1}{n}\right)\!\right)\!,\,\ \text{as}\ n\to\infty,\label{eq:hJ_n-asympt}
\end{equation}
where $\gamma\neq0$ is a constant. Without loss of generality we
can assume that $\gamma=1$. Moreover, referring to (\ref{eq:delta_1_asympt_jacobi}),
(\ref{eq:delta_3_asympt_jacobi}),
\begin{equation}
\lambda^{2}=\lim_{n\to\infty}\,\frac{h_{n+1}}{h_{n-1}}=\lim_{n\to\infty}\,\frac{h_{n+m+1}}{h_{n+m-1}}=\lim_{n\to\infty}\,\frac{\delta_{1}(n,m)}{\delta_{3}(n,m)}=\frac{\beta_{m}^{\text{J}}A(m)-\beta_{m+1}^{\text{J}}A(m-1)}{\beta_{m+1}^{\text{J}}A(m)-\beta_{m}^{\text{J}}A(m+1)}\,.\label{eq:J-lambda-2}
\end{equation}
In regard of (\ref{eq:delta_1_lead_m_asympt_jacobi}), (\ref{eq:delta_3_lead_m_asympt_jacobi}),
letting $m\to\infty$ we obtain $\lambda^{2}=1$.

Equation (\ref{eq:a_b_H}) in this case means that
\[
\big(A(n)-A(m)\big)h_{n+m+1}+(\beta_{n}^{\text{J}}-\beta_{m}^{\text{J}})h_{n+m}+\big(A(n-1)-A(m-1)\big)h_{n+m-1}=0.
\]
From here, when considering the limit $n\to\infty$ while taking into
account (\ref{eq:hJ_n-asympt}) and (\ref{eq:AJasympt}), (\ref{eq:betaJasympt}),
we get $A(m)\lambda^{2}+\beta_{m}^{\text{J}}\lambda+A(m-1)=0$. Hence
for $\lambda=\pm1$ we have
\[
\pm\beta_{m}^{\text{J}}+A(m)+A(m-1)=0,\,\ m\in\mathbb{N}_{0}.
\]
Referring again to (\ref{eq:AJasympt}), (\ref{eq:betaJasympt}) and
checking the asymptotic expansion, as $m\to\infty$, we find that
\[
\pm\beta_{m}^{\text{J}}+A(m)+A(m-1)=-\frac{(c\mp d)^{2}-1}{8\,m^{2}}+O\!\left(\frac{1}{m^{3}}\right)\!.
\]
Hence $(c\mp d)^{2}=1$. With this equation the asymptotic expansion
simplifies and we obtain
\[
\pm\beta_{m}^{\text{J}}+A(m)+A(m-1)=\mp\frac{3cd}{16\,m^{4}}+O\!\left(\frac{1}{m^{5}}\right)\!.
\]
Owing to our assumption $cd\neq0$ we arrive at a contradiction.

\vskip10pt 4)~\textbf{Case $c^{2}+d^{2}=1$, $cd\neq0$.} Discussion
of this case is similar to the foregoing one, but some leading terms
in the above asymptotic expansions disappear and this is why we have
to reconsider some formulas. First of all, we now have, as $n\to\infty$,
\[
A(n)=\frac{c^{2}d^{2}}{64\,n^{4}}-\frac{(c+2)\,c^{2}d^{2}}{32\,n^{5}}+O\!\left(\frac{1}{n^{6}}\right)\!,
\]
while the asymptotic expansion (\ref{eq:betaJasympt}) of $\beta_{n}^{\text{J}}$
remains as it is. Furthermore, for a fixed $m\in\mathbb{N}_{0}$,
\[
\delta_{2}(n,m)=-\det\!\left(\begin{array}{cc}
\alpha_{n}^{\text{J}}-\alpha_{m}^{\text{J}} & \alpha_{n-1}^{\text{J}}-\alpha_{m-1}^{\text{J}}\\
\alpha_{n-1}^{\text{J}}-\alpha_{m+1}^{\text{J}} & \alpha_{n-2}^{\text{J}}-\alpha_{m}^{\text{J}}
\end{array}\right)=A(m+1)A(m-1)-A^{2}(m)+O\!\left(\frac{1}{n^{4}}\right)\!,
\]
and the asymptotic expansions (\ref{eq:delta_1_asympt_jacobi}), (\ref{eq:delta_3_asympt_jacobi})
remain valid as they are. The following asymptotic expansions, as
$m\to\infty$, had to be reconsidered, however:
\begin{eqnarray}
A(m+1)A(m-1)-A^{2}(m) & = & \frac{c^{4}d^{4}}{1024\,m^{10}}+O\!\left(\frac{1}{m^{11}}\right)\!,\nonumber \\
\beta_{m}^{\text{J}}A(m)-\beta_{m+1}^{\text{J}}A(m-1) & = & -\frac{c^{3}d^{3}}{128\,m^{7}}+O\!\left(\frac{1}{m^{8}}\right)\!,\label{eq:J-delta-asympt-lead-1}\\
\beta_{m+1}^{\text{J}}A(m)-\beta_{m}^{\text{J}}A(m+1) & = & \frac{c^{3}d^{3}}{128\,m^{7}}+O\!\left(\frac{1}{m^{8}}\right)\!.\nonumber 
\end{eqnarray}

From (\ref{eq:J-delta-asympt-lead-1}) we conclude that in this case,
too, for all sufficiently large $m\in\mathbb{N}$ there exists $R_{m}\in\mathbb{N}$
such that for all $n\geq R_{m}$ and $j=1,2,3$, $\delta_{j}(n,m)\neq0$.
It follows that the asymptotic analysis (\ref{eq:hJ_n-asympt}) of
$h_{n}$, as $n\to\infty$, is still applicable. Whence, up to a constant
multiplier, $h_{n}=\lambda^{n}\big(1+O(n^{-1})\big)$ for some $\lambda\in\mathbb{R}$,
$\lambda\neq0$. We can again compute $\lambda^{2}$ as in (\ref{eq:J-lambda-2})
with the same result, namely
\[
\lambda^{2}=\frac{\beta_{m}^{\text{J}}A(m)-\beta_{m+1}^{\text{J}}A(m-1)}{\beta_{m+1}^{\text{J}}A(m)-\beta_{m}^{\text{J}}A(m+1)}\,.
\]
This time letting $m\to\infty$ and taking into account (\ref{eq:J-delta-asympt-lead-1})
we get $\lambda^{2}=-1$ which is not possible.

\subsection{Meixner--Pollaczek, Meixner, Laguerre, Charlier and Hermite}

The remaining cases will be directly treated using a necessary condition
for the existence of a nontrivial commuting Hankel matrix.

Shifting indices in the general commutation equation~\eqref{eq:a_b_H},
one gets the equation~\eqref{eq:Hprime} and further
\[
(\alpha_{n-2}-\alpha_{m+2})h_{n+m+1}+(\beta_{n-2}-\beta_{m+2})h_{n+m}+(\alpha_{n-3}-\alpha_{m+1})h_{n+m-1}=0.
\]
These three equations can have a nontrivial solution only if
\begin{equation}
D(m,n):=\det\begin{pmatrix}\alpha_{n}-\alpha_{m} & \beta_{n}-\beta_{m} & \alpha_{n-1}-\alpha_{m-1}\\
\alpha_{n-1}-\alpha_{m+1} & \beta_{n-1}-\beta_{m+1} & \alpha_{n-2}-\alpha_{m}\\
\alpha_{n-2}-\alpha_{m+2} & \beta_{n-2}-\beta_{m+2} & \alpha_{n-3}-\alpha_{m+1}
\end{pmatrix}=0,\label{eq:def_Dmn}
\end{equation}
for all $m,n\in\N$, $n\geq2$. In particular, if $\beta_{n}=0$ for
all $n\in\N_{0}$, it follows from~\eqref{eq:a_b_H} that the necessary
condition for existence of a nontrivial commuting Hankel matrix reads
\begin{equation}
\delta_{2}(m,n):=-\det\begin{pmatrix}\alpha_{n}-\alpha_{m} & \alpha_{n-1}-\alpha_{m-1}\\
\alpha_{n-1}-\alpha_{m+1} & \alpha_{n-2}-\alpha_{m}
\end{pmatrix}=0,\quad\forall m,n\in\N.\label{eq:def_dmn}
\end{equation}

\subsubsection{Meixner--Pollaczek}

Assume first that $\phi\neq\pi/2$. The leading terms of the asymptotic
expansions of $\alpha_{n}^{\text{MP}}$ and $\beta_{n}^{\text{MP}}$
are
\[
\alpha_{n}^{\text{MP}}=\frac{n}{2\sin\phi}+O(1)\quad\mbox{ and }\quad\beta_{n}^{\text{MP}}=-\frac{n}{\tan\phi}+O(1),\,\ \text{as}\ n\to\infty.
\]
Consequently, for $m\in\N$ fixed,
\[
D(m,n)=C_{m}(\lambda,\phi)n+O(1),\quad n\to\infty.
\]
By a simple linear algebra the function $C_{m}(\lambda,\phi)$ can
be readily computed,
\begin{align*}
C_{m}(\lambda,\phi)=\frac{\cos\phi}{4\sin^{3}\phi}\big[ & (\tilde{\alpha}_{m+1}^{\text{MP}}-\tilde{\alpha}_{m}^{\text{MP}})^{2}-(\tilde{\alpha}_{m+2}^{\text{MP}}-\tilde{\alpha}_{m+1}^{\text{MP}})(\tilde{\alpha}_{m}^{\text{MP}}-\tilde{\alpha}_{m-1}^{\text{MP}})\\
 & +\tilde{\alpha}_{m+2}^{\text{MP}}-3\tilde{\alpha}_{m+1}^{\text{MP}}+3\tilde{\alpha}_{m}^{\text{MP}}-\tilde{\alpha}_{m-1}^{\text{MP}}\big],
\end{align*}
where $\tilde{\alpha}_{n}^{\text{MP}}:=2\sin(\phi)\alpha_{n}^{\text{MP}}=\sqrt{(n+1)(n+2\lambda)}\,$.
Further, the asymptotic expansion of the above expression in the square
brackets equals
\[
-\frac{(1-2\lambda)^{4}}{32m^{6}}+O\!\left(\frac{1}{m^{7}}\right)\!,\,\ m\to\infty.
\]
Consequently, if $\Ch(\mathcal{J}^{\text{MP}})$ is nontrivial, then
$C_{m}(\lambda,\phi)=0$ for all $m\in\N$ which particularly implies
that $\lambda=1/2$.

If $\phi=\pi/2$, then $\beta_{n}^{\text{MP}}=0$ for all $n\in\N_{0}$,
and one proceeds similarly using~\eqref{eq:def_dmn} instead of~\eqref{eq:def_Dmn}.
The resulting formula reads
\[
\delta_{2}(m,n)=c_{m}(\lambda)n+O(1),\,\ \text{as}\ n\to\infty,
\]
where
\[
c_{m}(\lambda)=\frac{(1-2\lambda)^{2}}{16m^{3}}+O\!\left(\frac{1}{m^{4}}\right)\!,\,\ \text{as}\ m\to\infty.
\]
Thus one again concludes that $\lambda=1/2$, if $\Ch(\mathcal{J}^{\text{MP}})$
is nontrivial.

Assume on the other hand that $\lambda=1/2$. Note that $\lambda=1/2$
if and only if $\alpha_{n}^{\text{MP}}$ is a nonzero polynomial in
$n$. Then $\alpha_{-1}^{\text{MP}}=0$. Plugging $\alpha_{n}^{\text{MP}}$
and $\beta_{n}^{\text{MP}}$, given by~\eqref{eq:def_alpha_beta_meixner-pollaczek}
with $\lambda=1/2$, into \eqref{eq:a_b_H} we arrive at the second
order difference equation
\[
h_{k+1}-2\cos(\phi)h_{k}+h_{k-1}=0,\quad k\in\N,
\]
whose two linearly independent solutions are
\[
h_{k}^{(1)}=\sin(k\phi)\quad\mbox{ and }\quad h_{k}^{(2)}=\cos(k\phi).
\]
Consequently, $\dim\Ch(\mathcal{J}^{\text{MP}})=2$ and the two Hankel
matrices determined by the sequences $h^{(1)}$ and $h^{(2)}$ form
a basis of $\Ch(\mathcal{J}^{\text{MP}})$.

\subsubsection{Meixner}

In all remaining cases, the approach is completely analogous to the
case of the Meixner--Pollaczeck polynomials. Therefore we mention
only several most important points. In the case of the Meixner coefficients
defined by~\eqref{eq:def_alpha_beta_meixner} one obtains
\[
\lim_{n\to\infty}\frac{D(m,n)}{n}=\frac{c(c+1)(1-\beta)^{4}}{32(1-c)^{3}m^{6}}+O\!\left(\frac{1}{m^{7}}\right)\!,
\]
as $m\to\infty$. Since $c\in(0,1)$ necessarily $\beta=1$, if $\Ch(\mathcal{J}^{\text{M}})$
nontrivial.

Assuming on the other hand that $\beta=1$ and plugging~\eqref{eq:def_alpha_beta_meixner}
into~\eqref{eq:a_b_H}, one arrives at the three-term recurrence
\[
h_{k+1}+(c^{-1/2}+c^{1/2})h_{k}+h_{k-1}=0,\quad k\in\N,
\]
whose two linearly independent solutions are
\[
h_{k}^{(1)}=(-1)^{k}c^{k/2}\quad\mbox{ and }\quad h_{k}^{(2)}=(-1)^{k}c^{-k/2}.
\]

\subsubsection{Laguerre}

Using~\eqref{eq:def_alpha_beta_laguerre} in~\eqref{eq:def_Dmn},
one computes
\[
\lim_{n\to\infty}\frac{D(m,n)}{n}=\frac{\alpha^{4}}{16m^{6}}+O\!\left(\frac{1}{m^{7}}\right)\!,
\]
as $m\to\infty$. Hence $\Ch(\mathcal{J}^{\text{L}})$ can be nontrivial
only for $\alpha=0$. On the other hand, if $\alpha=0$, one deduces
from~\eqref{eq:a_b_H} the recurrence
\[
h_{k+1}+2h_{k}+h_{k-1}=0,\quad k\in\N,
\]
whose two linearly independent solutions are
\[
h_{k}^{(1)}=(-1)^{k}\quad\mbox{ and }\quad h_{k}^{(2)}=(-1)^{k}k.
\]

\subsubsection{Charlier}

Note that $\alpha_{n}^{\text{C}}$ defined in~\eqref{eq:def_alpha_beta_charlier}
is never a polynomial in $n$ for $a\neq0$. In that case one obtains
\[
\lim_{n\to\infty}\frac{D(m,n)}{n}=\frac{a}{8m^{3}}+O\!\left(\frac{1}{m^{4}}\right)\!,\,\ \text{as}\ m\to\infty.
\]
Since $a\neq0$ determinant $D(m,n)$ cannot vanish identically and
therefore $\Ch(\mathcal{J}^{\text{C}})$ is trivial.

\subsubsection{Hermite}

Substituting from~\eqref{eq:def_alpha_beta_hermite} into the equation~\eqref{eq:a_b_H},
one obtains the second order recurrence
\[
h_{k+1}+h_{k-1}=0,\quad k\in\N,
\]
whose two linearly independent solutions are
\[
h_{k}^{(1)}=\sin\frac{k\pi}{2}\quad\mbox{ and }\quad h_{k}^{(2)}=\cos\frac{k\pi}{2}.
\]

\section*{Acknowledgments}

The authors acknowledge financial support by the Ministry of Education,
Youth and Sports of the Czech Republic project no. CZ.02.1.01/0.0/0.0/16\_019/0000778.


\end{document}